\documentclass[11pt]{article}
\usepackage[centertags]{amsmath}
\usepackage{amsfonts}
\usepackage{amssymb}
\usepackage{amsthm,color}
\usepackage{newlfont}

\newtheorem{Theorem}{Theorem}[section]
\newtheorem{Definition}[Theorem]{Definition}
\newtheorem{Proposition}[Theorem]{Proposition}
\newtheorem{Lemma}[Theorem]{Lemma}

\newtheorem{Remark}[Theorem]{Remark}
\newtheorem{Example}[Theorem]{Example}

\newtheorem{Hypothesis}{Hypothesis}

\numberwithin{equation}{section}

\begin{document}

\def\le{\left}
\def\r{\right}
\def\cost{\mbox{const}}
\def\a{\alpha}
\def\d{\delta}
\def\ph{\varphi}
\def\e{\epsilon}
\def\la{\lambda}
\def\si{\sigma}
\def\La{\Lambda}
\def\B{{\cal B}}
\def\A{{\mathcal A}}
\def\L{{\mathcal L}}
\def\O{{\mathcal O}}
\def\bO{\overline{{\mathcal O}}}
\def\F{{\mathcal F}}
\def\K{{\mathcal K}}
\def\H{{\mathcal H}}
\def\D{{\mathcal D}}
\def\C{{\mathcal C}}
\def\M{{\mathcal M}}
\def\N{{\mathcal N}}
\def\G{{\mathcal G}}
\def\T{{\mathcal T}}
\def\R{{\mathbb R}}
\def\I{{\mathcal I}}

\def\bw{\overline{W}}
\def\phin{\|\varphi\|_{0}}
\def\s0t{\sup_{t \in [0,T]}}
\def\lt{\lim_{t\rightarrow 0}}
\def\iot{\int_{0}^{t}}
\def\ioi{\int_0^{+\infty}}
\def\ds{\displaystyle}
\def\pag{\vfill\eject}
\def\fine{\par\vfill\supereject\end}
\def\acapo{\hfill\break}

\def\beq{\begin{equation}}
\def\eeq{\end{equation}}
\def\barr{\begin{array}}
\def\earr{\end{array}}
\def\vs{\vspace{.1mm}   \\}
\def\rd{\reals\,^{d}}
\def\rn{\reals\,^{n}}
\def\rr{\reals\,^{r}}
\def\bD{\overline{{\mathcal D}}}
\newcommand{\dimo}{\hfill \break {\bf Proof - }}
\newcommand{\nat}{\mathbb N}
\newcommand{\E}{\mathbb E}
\newcommand{\Pro}{\mathbb P}
\newcommand{\com}{{\scriptstyle \circ}}
\newcommand{\reals}{\mathbb R}

\title{Pathwise uniqueness for stochastic reaction-diffusion equations in Banach
spaces with an H\"{o}lder drift component}

\author{Sandra Cerrai\thanks{Partially supported by the NSF grant DMS0907295 ``Asymptotic Problems for SPDE's''.}
\\
\normalsize University of Maryland, College Park, USA
\and
Giuseppe Da Prato\\
\normalsize Scuola Normale Superiore,  Pisa, Italy
\and
Franco Flandoli\\
\normalsize Universit\`a di Pisa, Italy
}

\date{}

\maketitle

\begin{abstract}
We prove pathwise uniqueness for an abstract  stochastic reaction-diffusion
equation in Banach spaces. The drift contains a bounded H\"{o}lder term; in
spite of this, due to the space-time white noise it is possible to prove
pathwise uniqueness. The proof is based on a detailed analysis of the
associated Kolmogorov equation. The model includes examples not covered by the
previous works based on Hilbert spaces or concrete SPDEs.

\end{abstract}

We prove pathwise uniqueness for a general class of reaction-diffusion
equations in Banach spaces with an H\"{o}lder drift component, of the form%
\[
\left\{
\begin{array}
[c]{l}%
dX(t)=[AX(t)+F(X(t))+B(X(t))]dt+dw(t)\\
\\
X(0)=x.
\end{array}
\right.
\]
Here $A$ is the Laplacian operator in the 1-dimensional space domain $\left[
0,1\right]  $ with Dirichlet or Neumann boundary conditions, the Banach space
$E$ is the closure of $D(A)$ in $C\left(  \left[  0,1\right]  \right)  $,
$x\in E$, $F$ is a very general reaction-diffusion operator in $E$ which
covers the usual polynomial nonlinearities with odd degree, having strictly negative leading coefficient, 
$B:E\rightarrow E$ is only H\"{o}lder continuous and bounded, $w\left(
t\right)  $ is a space-time white noise. See the next section for more details,
in particular about the assumptions on $F$.

For finite dimensional stochastic differential equations it is well known that
additive non degenerate noise leads to pathwise uniqueness in spite of the
poor regularity of the drift (see \cite{V81}, \cite{KR05} among others). Due
to a number of relevant open problems of uniqueness for PDEs, there is intense
research activity to understand when noise improves uniqueness in infinite
dimensions (see \cite{F11} for a review). Our result, which applies to a large
class of systems of interest for applications, contributes to this research direction.

The present paper is the first one dealing with the problem of pathwise uniqueness in  Banach spaces instead of Hilbert spaces.
This extension introduces many difficulties and does not represent a mere generalization of the previous cases studied in the existing literature. We treat here the concrete case of the
Banach space $E=C\left(  \left[  0,1\right]  \right)  $ or $E=C_{0}\left(
\left[  0,1\right]  \right)  $ (depending on the boundary conditions). A
typical tool in Hilbert spaces is the finite dimensional projection or
approximation by means of the elements of an orthonormal basis. Here we
implement the idea recently developed in \cite{CD12} of using an orthonormal
basis of the Hilbert space $L^{2}\left(  0,1\right)  $ made of elements which
belong to $E$. This method allows to perform certain finite dimensional
approximations and in particular to write It\^{o} formulae for certain
quantities;\ the control of many terms is often nontrivial but successful.

This paper is, in a sense, the generalization of \cite{DF10} to Banach spaces
(see also \cite{DFPR12} on bounded measurable drift and the work in finite
dimensions \cite{FGP10} where part of the technique was developed in order to
construct stochastic flows). From the viewpoint of examples, this
generalization is relevant. Both the reaction-diffusion term $F$ and the
H\"{o}lder term $B$ are not covered by \cite{DF10} except for particular
cases. One could na\"{\i}vely think that it is sufficient to apply a cut-off
and reduce (locally in time) reaction diffusion to the Hilbert set-up but it
is not so: a cut-off of the form $\varphi\left(  \left\Vert x\right\Vert
_{L^{2}}\right)  $ does not make a polynomial $x^{n}$ locally Lipschitz in
$L^{2}$. Concerning the H\"{o}lder term $B$, there are examples in $E$ which
are not even defined as operators on $L^{2}\left(  0,1\right)  $, see section
\ref{subsection examples}.

Before the more recent works (the present one and the other mentioned above)
on pathwise uniqueness for abstract stochastic evolution equations in Hilbert
or Banach spaces, there have been several important works on one-dimensional
SPDEs of parabolic type driven by space-time white noise, with several levels
of generality of the drift term, see \cite{GP93}, \cite{G98}, \cite{GN99},
\cite{GM01}, \cite{AG01}. These works remain highly competitive with the
abstract ones, and sometimes more general, but conversely the abstract works
cover examples not treated there. Concerning reaction-diffusion, some examples
are included in these previous works but not in the generality treated here
and moreover, the abstract nature of the H\"{o}lder term $B$ allows us to
treat new examples, like those of section \ref{subsection examples}.

Finally, we want to stress that this paper contains, for the purpose of
pathwise uniqueness, a detailed analysis of the Kolmogorov equation associated
to the SPDE above. These results may have other applications and also an
intrinsic interest for infinite dimensional analysis. The Kolmogorov equation
associated to reaction-diffusion equation has been investigated in \cite{C01},
\cite{C03}, \cite{CD12} and related works. In our work here we add new
informations. First, an improved analysis of second derivatives is given,
needed to control one of the terms which appears in the reformulated evolution
equation (one of the main points for the proof of pathwise uniqueness).
Second, a vectorial form of the Kolmogorov equation is discussed, again needed
in this particular approach to pathwise uniqueness. Third, the classical case
of the Kolmogorov equation with reaction diffusion term $F$ has been extended
to cover also the H\"{o}lder operator $B$.

\subsection{Examples\label{subsection examples}}

Let $E=C\left(  \left[  0,1\right]  \right)  $, $H=L^{2}\left(  0,1\right)  $.
We give two examples of maps $B:E\rightarrow E$ which are not well defined as
maps from $H$ to $H$, and are of class%
\[
B\in C_{b}^{\alpha}\left(  E,E\right)  .
\]
This shows that our theory has more applications than the previous works.

\begin{Example}
{\em 
Given $g\in E$, $\xi_{0}\in\left[  0,1\right]  $, $b\in C_{b}^{\alpha}\left(
\mathbb{R},\mathbb{R}\right)  $ such that
\[
\left\vert b\left(  s\right)  -b\left(  s^{\prime}\right)  \right\vert \leq
M\left\vert s-s^{\prime}\right\vert ^{\alpha}
\]
set%
\[
B\left(  x\right)  \left(  \xi\right)  =b\left(  x\left(  \xi_{0}\right)
\right)  g\left(  \xi\right),\ \ \ \ x \in\, E.
\]
Then $B\in C_{b}^{\alpha}\left(  E,E\right)  $. Indeed,%
\[\begin{array}{l}
\ds{| B\left(  x\right)  -B\left(  x^{\prime}\right) | _{E}  
=\max_{\xi\in\left[  0,1\right]  }\left\vert b\left(  x\left(  \xi_{0}\right)
\right)  g\left(  \xi\right)  -b\left(  x^{\prime}\left(  \xi_{0}\right)
\right)  g\left(  \xi\right)  \right\vert }\\
\vs
\ds{
=\left\vert b\left(  x\left(  \xi_{0}\right)  \right)  -b\left(  x^{\prime
}\left(  \xi_{0}\right)  \right)  \right\vert | g| _{E}
 \leq M\left\vert x\left(  \xi_{0}\right)  -x^{\prime}\left(  \xi_{0}\right)
\right\vert ^{\alpha}|g| _{E}\leq M\left\Vert g\right\Vert
_{E}|x-x^{\prime}| _{E}^{\alpha}.}
\end{array}\]

}

\end{Example}

\begin{Example}
{\em 
Given $b$ as above, set%
\[
B\left(  x\right)  \left(  \xi\right)  =b\left(  \max_{s\in\left[
0,\xi\right]  }x\left(  s\right)  \right)  .
\]
Then $B\in C_{b}^{\alpha}\left(  E,E\right)  $. Indeed,%
\[\begin{array}{l}
\ds{|B\left(  x\right)  -B\left(  x^{\prime}\right) | _{E}  
=\max_{\xi\in\left[  0,1\right]  }\left\vert b\left(  \max_{s\in\left[
0,\xi\right]  }x\left(  s\right)  \right)  -b\left(  \max_{s\in\left[
0,\xi\right]  }x^{\prime}\left(  s\right)  \right)  \right\vert }\\
\vs
\ds{ \leq M\max_{\xi\in\left[  0,1\right]  }\left\vert \max_{s\in\left[
0,\xi\right]  }x\left(  s\right)  -\max_{s\in\left[  0,\xi\right]  }x^{\prime
}\left(  s\right)  \right\vert ^{\alpha}.}
\end{array}\]
Now, one has
\begin{equation}
\left\vert \max_{s\in\left[  0,\xi\right]  }x\left(  s\right)  -\max
_{s\in\left[  0,\xi\right]  }x^{\prime}\left(  s\right)  \right\vert \leq
\max_{s\in\left[  0,\xi\right]  }\left\vert x\left(  s\right)  -x^{\prime
}\left(  s\right)  \right\vert \label{max ineq}%
\end{equation}
Indeed, assume that 
\[\max_{s\in\left[  0,\xi\right]  }x\left(  s\right)
\geq\max_{s\in\left[  0,\xi\right]  }x^{\prime}\left(  s\right).\]
 Let
$s_{M},s_{M}^{\prime}\in\left[  0,\xi\right]  $ be two points such that
\[\max_{s\in\left[  0,\xi\right]  }x\left(  s\right)  =x\left(  s_{M}\right),\ \ \ \ 
\max_{s\in\left[  0,\xi\right]  }x^{\prime}\left(  s\right)  =x^{\prime
}\left(  s_{M}^{\prime}\right).\] We have%
\[
x^{\prime}\left(  s_{M}^{\prime}\right)  \geq x^{\prime}\left(  s_{M}\right)
\]
and thus%
\[\begin{array}{l}
\ds{\max_{s\in\left[  0,\xi\right]  }x\left(  s\right)  -\max_{s\in\left[
0,\xi\right]  }x^{\prime}\left(  s\right) =x\left(  s_{M}\right)
-x^{\prime}\left(  s_{M}^{\prime}\right)  \leq x\left(  s_{M}\right)
-x^{\prime}\left(  s_{M}\right)}\\
\vs
\ds{
\leq\max_{s\in\left[  0,\xi\right]  }\left(  x\left(  s\right)  -x^{\prime
}\left(  s\right)  \right)  \leq\max_{s\in\left[  0,\xi\right]  }\left\vert
x\left(  s\right)  -x^{\prime}\left(  s\right)  \right\vert .}
\end{array}\]
We arrive to the same conclusion if $\max_{s\in\left[  0,\xi\right]  }x\left(
s\right)  \leq\max_{s\in\left[  0,\xi\right]  }x^{\prime}\left(  s\right)
$.Therefore we have proved (\ref{max ineq}). We apply it to the estimates
above and get%
\[\begin{array}{l}
\ds{| B\left(  x\right)  -B\left(  x^{\prime}\right)  | _{E}  
\leq M\max_{\xi\in\left[  0,1\right]  }\left(  \max_{s\in\left[  0,\xi\right]
}\left\vert x\left(  s\right)  -x^{\prime}\left(  s\right)  \right\vert
\right)  ^{\alpha}}\\
\vs
\ds{=M\max_{\xi\in\left[  0,1\right]  }\max_{s\in\left[  0,\xi\right]
}\left\vert x\left(  s\right)  -x^{\prime}\left(  s\right)  \right\vert
^{\alpha}=M\,|x-x^{\prime}| _{E}^{\alpha}.}
\end{array}\]
The proof is complete.}
\end{Example}

\begin{Example}
{\em With minor adjustments the same result holds for%
\[
B\left(  u\right)  \left(  \xi\right)  =b\left(  \max_{s\in\left[
0,\xi\right]  }\left\vert u\left(  s\right)  \right\vert \right)  .
\]}

\end{Example}

\begin{Remark}
{\em On the contrary, the example%
\[
B\left(  u\right)  \left(  \xi\right)  =b\left(  u\left(  \xi\right)  \right)
\]
is also of class $B\in C_{b}^{\alpha}\left(  H,H\right)  $ and thus it is
covered by the previous theories. Indeed,
\[\begin{array}{l}
\ds{\left\Vert B\left(  u\right)  -B\left(  u^{\prime}\right)  \right\Vert
_{H}^{2}   =\int_{0}^{1}\left\vert b\left(  u\left(  \xi\right)  \right)
-b\left(  u^{\prime}\left(  \xi\right)  \right)  \right\vert ^{2}d\xi}\\
\vs
\ds{\leq M^{2}\int_{0}^{1}\left\vert u\left(  \xi\right)  -u^{\prime}\left(
\xi\right)  \right\vert ^{2\alpha}d\xi \leq M^{2}\left(  \int_{0}^{1}\left\vert u\left(  \xi\right)  -u^{\prime
}\left(  \xi\right)  \right\vert ^{2}d\xi\right)  ^{\alpha}=M^{2}\left\Vert
u-u^{\prime}\right\Vert _{H}^{2\alpha}.}
\end{array}\]}

\end{Remark}

\subsection{Notations}

Let $X$ and $Y$ be two separable Banach spaces. In what follows, we shall denote by $B_b(X,Y)$ the Banach space of bounded Borel function $\varphi:X\to Y$, endowed with the sup-norm
 \[\|\varphi\|_{B_b(X,Y)}:=\sup_{x \in\,X}|\varphi(x)|_Y,\]
and by $C_b(X,Y)$ the subspace of uniformly continuous mappings.  $\text{Lip}_{b}(X,Y)$ is the subspace of Lipschitz-continuous mappings, endowed with the norm
\[\|\varphi\|_{{\text{\tiny Lip}_b}(X,Y)}:=\|\varphi\|_{C_b(X,Y)}+\sup_{x,y\in E,\,x\neq y}\frac{|\varphi(x)-\varphi(y)|_Y}{|x-y|_X}=:\|\varphi\|_{C_b(X,Y)}+[\varphi]_{{\text{\tiny Lip}_b}(X,Y)}.\]
For any $\theta \in\,(0,1)$, we denote by $C^\theta_b(X,Y)$ the Banach space of all $\theta$-H\"older continuous mappings $\varphi\in C_b(X,Y)$, endowed with the norm
$$
\|\varphi\|_{C^\theta(X,Y)}=\|\varphi\|_{C_b(X,Y)}+\sup_{x,y\in X,\,x\neq y}\frac{|\varphi(x)-\varphi(y)|_Y}{|x-y|_X^\theta}.
$$

Finally, for any integer $k\geq1$, we denote by $C^k_b(X,Y)$ the space of all mappings $\varphi:X\to Y$ which are $k$ times differentiable, with uniformly continuous and bounded derivatives. $C^k_b(X,Y)$ is a Banach space, endowed with the norm
 \[\|\varphi\|_{C^k_b(X,Y)}=:\|\varphi\|_{C_b(X,Y)}+\sum_{j=1}^k\sup_{x \in\,X}\|D^j\varphi(x)\|_{{\cal L}^j(X,Y)}.\]
Spaces $C^{\theta+k}_b(X,Y)$, with $k\in\mathbb{N}$ and $\theta \in\,(0,1)$, are defined similarly.

Finally,  when $Y=\reals$, we shall denote $B_b(X,Y)$ and $C_b^{\theta+k}(X,Y)$, for $\theta \in\,[0,1]$ and $k \in\,\nat$, by $B_b(X)$ and $C_b^{\theta+k}(X)$, respectively.

\bigskip

\section{The unperturbed reaction-diffusion equation} 
\label{sec2}
 We are here concerned with the following stochastic reaction--diffusion equation in the Banach space $C([0,1])$,
\begin{equation}
\label{e1.1}
\left\{\begin{array}{lll}
dX(t,\xi)=[D^2_\xi X(t,\xi)+f(\xi,X(t,\xi))]dt+dw(t,\xi),\quad \xi\in (0,1),\\
\\
{\cal B}X(t,0)={\cal B}X(t,1)=0,\quad t\ge 0,\\
\\
X(0,\xi)=x(\xi),\quad \xi\in [0,1],
\end{array}\right.
\end{equation}
where $b:[0,1]\times \R\to\R$ is a given function, $w(t)$  is a cylindrical Wiener process in $L^2(0,1)$, defined on a filtered
probability space
$(\Omega,\mathcal F, (\mathcal F_t)_{t\ge 0},\Pro)$, and  either ${\cal B}u=u$ (Dirichlet boundary condition) or ${\cal B}u=u^\prime$ (Neumann boundary condition).

If we denote by $A$ the  realization in $C([0,1])$ of the operator $D^2_\xi$, endowed with the boundary condition ${\cal B}$,
and if we denote by $F$ the Nemytski operator associated with $f$, namely
\[F(x)(\xi)=f(\xi,x(\xi)),\ \ \ \ x \in\,C([0,1]),\ \ \ \xi \in\,[0,1],\]
then problem \eqref{e1.1} can be written as the following stochastic differential equation
  in $C([0,1])$
\begin{equation}
\label{e1.3}
\left\{\begin{array}{lll}
dX(t)=[A X(t)+F(X(t))]dt+dw(t),\\
\\
X(0)=x.
\end{array}\right.
\end{equation}

In what follows, we shall denote by $H$ the Hilbert space $L^2(0,1)$, endowed with the scalar product $\langle\cdot,\cdot\rangle_H$ and the corresponding norm $|\cdot|_H$.
With $E$ we shall denote the closure of $D(A)$ in the space $C([0,1])$, endowed with the uniform norm $|\cdot|_E$ and the duality $\le<\cdot,\cdot\r>_E$  between $E$ and $E^\star$. Notice that in the case of Dirichlet boundary conditions $\overline{D(A)}=C_0([0,1])$ and in the case of Neumann boundary conditions $\overline{D(A)}=C([0,1])$.  However, in both cases
the semigroup  $e^{tA}$ generated by $A$  is strongly continuous and analytic in $E$.
 Finally, for any $\e>0$ we shall denote by $E_\e$ the subspace of $\e$-H\"older continuous functions,  endowed with the norm 
\[|x|_{E_\e}:=|x|_E+\sup_{\substack{\xi,\eta \in\,[0,1]\\ \xi\neq \eta}}\frac{|x(\xi)-x(\eta)|}{|\xi-\eta|^\e_E}.\]

\medskip

In what follows, we shall assume that the mapping $f:[0,1]\times \R\to \R$ is continuous and satisfies the following conditions.
\begin{Hypothesis}
\label{H1}
\begin{enumerate}
\item For any $\xi \in\,[0,1]$, the mapping $f(\xi,\cdot):\R\to \R$ is of class $C^3$ and there exists an integer $m\geq 0$ such that
\begin{equation}
\label{dp3}
\sup_{\substack{\xi \in\,[0,1]\\s \in\,\R}}\,\frac{|D^j_s  f(\xi,s)|}{1+|s|^{2m+1-j}}<\infty,\quad j=0,1,2,3.
\end{equation}
Moreover, the mappings $D_s^jf:[0,1]\times \R\to \R$ are all continuous.
\item We have
\begin{equation}
\label{dp33}
\sup_{\substack{\xi \in\,[0,1]\\s \in\,\R}}\,Df(\xi,s)=:\rho<\infty.
\end{equation}
\item If $m\geq 1$, then there exist $\a>0$, $\gamma\geq 0$ and $c \in\,\R$ such that
\[\sup_{\xi \in\,[0,1]}\le(f(\xi,s+h)-f(\xi,s)\r)h\leq -\a h^{2(m+1)}+c\,\le(1+|s|^{\gamma}\right).\]
\end{enumerate}

\end{Hypothesis}

A simple example of a function $f$ fulfilling all conditions in Hypothesis \ref{H1}
is
\[f(\xi,s)=-\a(\xi)\,s^{2m+1}+\sum_{j=0}^{2m} c_j(\xi)\,s^j,\ \ \ \ \ (\xi,s) \in\,[0,1]\times \R,\]
for some continuous functions $\a, c_j:[0,1] \to \mathbb{R}$, with
\[\inf_{\xi \in\,[0,1]}\a(\xi)=:\a_0>0.\]

\begin{Definition}
Let $x\in E.$ We say that an adapted process $X(\cdot,x)$ is a
 {\em mild} solution
of problem \eqref{e1.1} if $X(t,x)\in E$, for all $t\ge 0$, and fulfills the integral equation
\begin{equation}
\label{e1.4}
X(t,x)=e^{tA}x+\int_0^te^{(t-s)A}F(X(s))ds+W_A(t),\quad t\ge 0,
\end{equation}
where $W_A(t)$ is the stochastic convolution
$$
W_A(t)=\int_0^te^{(t-s)A}dw(s),\quad t\ge 0.
$$
\end{Definition} 

In \cite[Proposition 6.2.2]{C01} is proved that, for any $x \in\,E$, problem \eqref{e1.1} admits a unique adapted {\em mild solution} $X(\cdot,x) \in\,L^p(\Omega;C([0,T];E))$, for any $T>0$ and $p\geq 1$, such that for any $t \in\,[0,T]$
\begin{equation}
\label{6.2.2}
\sup_{s \in\,[0,t]}|X(s,x)|_E\leq \Lambda(t)\,\le(1+|x|_E\r),\ \ \ \ \Pro-\text{a.s.}
\end{equation}
for some  random variable $\Lambda(t)$ such that
\[\E\,\Lambda(t)^p<\infty,\]
for any $p\geq 1$.
Moreover, in \cite[Theorem 6.2.3]{C01} is proved that for any $t>0$
\begin{equation}
\label{6.2.9}
\sup_{x \in\,E}|X(t,x)|_E\leq \Gamma(t)\,t^{-\frac 1{2m}},\ \ \ \ \ \Pro-\text{a.s.}
\end{equation}
for some random variable $\Gamma(t)$, increasing with respect to $t$,  such that
\[\E\,\Gamma(t)^p<\infty,\]
for any $p\geq 1$ and $t\geq 0$.

Notice that there exists $\e_0>0$ such that for any $x \in\,E$
\begin{equation}
\label{dp161}
X(t,x) \in\,E_{\e_0},\ \ \ \    t>0,\ \ \ \Pro-\text{a.s.}
\end{equation}
and the mapping $x \in\,E\to X(t,x) \in\,E_{\e_0}$ is continuous, $\Pro$-a.s.

Moreover, in \cite[Proposition 7.1.2]{C01} it has been proved that for any $x \in\,H$ there exists a unique {\em generalized solution} $X(\cdot,x) \in\,L^p(\Omega;C([0,T];H))$, for any $p\geq 1$ and $T>0$. This means that for any sequence $\{x_n\}_{n \in\,\nat}\subset E$ converging to $x$ in $H$, the sequence $\{X(\cdot,x_n)\}_{n  \in\,\nat}$ converges to $X(\cdot,x)$ in $C([0,T];H)$, $\Pro$-a.s. as $n\rightarrow \infty$.
Furthermore, estimates analogous to \eqref{6.2.2} and \eqref{6.2.9} hold in $H$. Namely,
\begin{equation}
\label{6.2.2H}
\sup_{s \in\,[0,t]}|X(s,x)|_H\leq \Lambda(t)\,\le(1+|x|_H\r),\ \ \ \ \Pro-\text{a.s.}
\end{equation}
and 
\begin{equation}
\label{6.2.9H}
\sup_{x \in\,H}|X(t,x)|_H\leq \Gamma(t)\,t^{-\frac 1{2m}},\ \ \ \ \ \Pro-\text{a.s.}
\end{equation}
for suitable random variables $\Lambda(t)$ and $\Gamma(t)$ as above.

In \cite[Chapter 6]{C01} the regularity of the mapping
\[x \in\,E\mapsto X(t,x) \in\,C([0,T];L^p(\Omega,E)),\]
has been studied and in Theorem 6.3.3 it has been proved that, as $f$ is assumed to be of class $C^3$, such a mapping is three times differentiable and the derivatives satisfy
\begin{equation}
\label{6.3.9}
\sup_{\substack{x \in\,E\\ t \in\,[0,T]}}|D^j_xX(t,x)(h_1,\ldots,h_r)|_E\leq \Lambda_j(T)|h_1|_E\cdots|h_r|_E,
\end{equation}
for any $r=1,2,3$, $T>0$ and $h_1,\ldots,h_r \in\,E$ and for some random variables $\Lambda_j(T)$ having  finite moments of any order.

The regularity of the mapping
\[x \in\,H\mapsto X(t,x) \in\,C([0,T];L^p(\Omega,H)),\]
has not been investigated, but in \cite[Proposition 7.2.1]{C01} it has been proved that for any $x, h \in\,H$ there exists a process $v(\cdot,x,h)$ such that for any two sequences $\{x_n\}_{n \in\,\nat}$ and $\{h_n\}_{n \in\,\nat}$, converging in $H$ to $x$ and $h$, respectively, the sequence $\{D_x X(\cdot,x_n)h_n\}_{n \in\,\nat}$ converges to $v(\cdot,x,h)$ in $C([0,T];H)$, $\Pro$-a.s.

Now, for any $x \in\,E$, $h \in\,H$  and $s\geq 0$, let us consider the problem
\begin{equation}
\label{random}
\eta^\prime(t)=A\eta(t)+F^\prime(X(t,x))\eta(t),\ \ \ \ \ \eta(s)=h,\ \ t\geq s,\ \ \ \omega \in\,\Omega.
\end{equation}
This is a random equation, whose solution is denoted by $\eta(t;s,x,h)$, and 
it defines the following random evolution operator
\begin{equation}
\label{dp225}
[U^x_{t,s}(\omega)]h=\eta(t;s,x,h)(\omega).\end{equation}

In view of Hypothesis \ref{H1}, it is immediate to check that  $U^x_{t,s}$ satisfies the following properties (the proof is left to the reader).
\begin{Lemma}
\label{l2.2}
\begin{enumerate}
\item There exists a kernel $K^x_{t,s}:\Omega\times [0,1]\times[0,1]\to \reals^+$ such that for any 
$x \in\,E$, $h \in\,H$ and $0\leq s\leq t$
\begin{equation}
\label{dp101}
U^x_{t,s}(\omega)h(\xi)=\int_0^1 K^x_{t,s}(\omega,\xi,\theta)h(\theta)\,d\theta.
\end{equation}
\item For any $(\xi,\theta) \in\,[0,1]\times[0,1]$ and $x \in\,E$, we have
\begin{equation}
\label{dp102}
0\leq K^x_{t,s}(\omega,\xi,\theta)\leq K_{t-s}(\xi,\theta)\,e^{\rho t}\leq (4\pi t)^{-\frac 12}\,e^{\rho t},\ \ \ \ 0\leq s\leq t,\ \ \ \Pro-\text{a.s.}
\end{equation}
where $\rho$ is the constant introduced in \eqref{dp33} and $K_t(\xi,\theta)$ is the kernel associated with the operator $A$.
\item The evolution operator $U^x_{t,s}$ is ultracontractive and for any $1\leq q\leq p$
\begin{equation}
\label{dp103}
|U^x_{t,s}(\omega)h|_p\leq c((t-s)\wedge 1)^{-\frac{p-q}{2pq}}|h|_q,\ \ \ \ \ t>s,\ \ \ \Pro-\text{a.s.}
\end{equation}
\end{enumerate}

\end{Lemma}

As a consequence of the previous Lemma,  the following fact holds.
\begin{Lemma} 
\label{l23}
We have
\begin{equation}
\label{dp100}
\sup_{x \in\,E}\,\sum_{i=1}^\infty |D_xX(t,x)e_i|_H^2\leq c\,e^{2\rho t}\,t^{-\frac 12},\ \ \ \ t >0,\ \ \ \ \Pro-\text{a.s.}
\end{equation}
for some constant $c>0$. Moreover, the sum converges uniformly with respect to $x \in\,E$.
\end{Lemma}

\begin{proof}
We have
$D_xX(t,x)e_i=U^x_{t,0}e_i$, hence, due to \eqref{dp101}, we have
\[\begin{array}{l}
\ds{\sum_{i=1}^\infty|D_xX(t,x)e_i(\xi)|^2=\sum_{i=1}^\infty\le|\le<K^x_{t,0}(\xi,\cdot),e_i\r>_H\r|^2=|K^x_{t,0}(\xi,\cdot)|_H^2\leq |K_{t}(\xi,\cdot)|_H^2 e^{2\rho t}.}\end{array}\]
This implies that for any $t>0$
\[\begin{array}{l}
\ds{\sum_{i=1}^\infty|D_xX(t,x)e_i|_H^2\leq \int_0^1|K_{t}(\xi,\cdot)|_H^2\,d\xi\, e^{2\rho t}\leq c\,e^{2\rho t}\,t^{-\frac 12}.}\end{array}\]

\end{proof}

\begin{Remark}
{\em  Due to \eqref{dp103}, for any $1\leq p\leq q\leq \infty$ we have
\begin{equation}
\label{dp120}
|D_xX(t,x)v|_p\leq c\,(t\wedge 1)^{-\frac{p-q}{2pq}}|v|_q.
\end{equation}
In particular, if $x,y \in\,E$, we have
\[|X(t,x)-X(t,y)|_H\leq \int_0^1|D_xX(t,\theta x+(1-\theta)y)(x-y)|_H\,d\theta\leq c(t) |x-y|_H.\]
Recalling how the generalized solution $X(t,x)$ has been constructed in $H$, this implies that for any $x,y \in\,H$ and $t\geq 0$
\begin{equation}
\label{dp123}
|X(t,x)-X(t,y)|_H\leq  c(t) |x-y|_H.
\end{equation}
}
\end{Remark}

\medskip

\begin{Lemma}
\label{dp110}
For any $x,h \in\,E$ and $t\geq 0$, we have
\begin{equation}
\label{dp109}
\sum_{i=1}^\infty |D^2_xX(t,x)(e_i,e_i)|_H\leq \kappa(t)\,\le(1+|x|_E^{2m-1}\r)\,(t\wedge 1)^{\frac 14},\ \ \ \ \Pro-\text{a.s.}
\end{equation}
for some positive random variable $\kappa(t)$, increasing with respect to $t\geq0$, having all moments finite.
Moreover,  the sum converges uniformly with respect to $x \in\,B_E(R)$, for any $R>0$.

\end{Lemma}
\begin{proof}
We have
\[D^2_xX(t,x)(e_i,e_i)=\int_0^tU^x_{t,s}F^{\prime \prime}(X(s,x))(D_x X(s,x)e_i,D_x X(s,x)e_i)\,ds.\]
Then, 
due to \eqref{6.2.2}, \eqref{dp103} and \eqref{dp120}, we have
\[\begin{array}{l}
\ds{|D^2_xX(t,x)(e_i,e_i)|_H\leq c\int_0^t((t-s)\wedge 1)^{-\frac 14}|F^{\prime \prime}(X(s,x))|_E
|D_x X(s,x)e_i|^2_H\,ds}\\
\vs
\ds{\leq c\, e^{\rho t}\int_0^t((t-s)\wedge 1)^{-\frac 14}\le(1+|X(s,x))|_E^{2m-1}\r)
|D_x X(s,x)e_i|^2_H\,ds}\\
\vs
\ds{\leq c\, e^{\rho t}\,\Lambda(t)^{2m-1}\le(1+|x|_E^{2m-1}\r)\int_0^t((t-s)\wedge 1)^{-\frac 14}
|D_x X(s,x)e_i|^2_H\,ds.}
\end{array}\]
Thanks to \eqref{dp100}, this implies
\[\begin{array}{l}
\ds{\sum_{i=1}^\infty|D^2_xX(s,x)(e_i,e_i)|_H\,ds\leq c\, e^{3\rho t}\,\Lambda(t)^{2m-1}_H\le(1+|x|_E^{2m-1}\r)\int_0^t((t-s)\wedge 1)^{-\frac 14}s^{-\frac 12}\,ds,}
\end{array}\]
and \eqref{dp109} follows.

\end{proof}

\begin{Remark}
\label{rem2.6}
{\em Let $J_n=nR(n,A)$. Then, from the proof above, we have also that 
\begin{equation}
\label{dp108}
\sum_{i=1}^\infty |D^2_xX(t,J_nx)(J_ne_i,J_ne_i)|_H\leq \kappa(t)\,\le(1+|x|_E^{2m-1}\r).
\end{equation}
Notice that
 the series converges uniformly with respect to $n \in\,\nat$ and $x \in\,B_E(R)$.

}
\end{Remark}

\begin{Lemma}
\label{l2.7}
For any $x,y \in\,E$ and $h \in\,H$ and $t>0$, we have
\begin{equation}
\label{dp125}
|D_xX(t,x)h-D_xX(t,y)h|_H\leq \kappa(t)\,|x-y|_E|h|_H (t\wedge 1)^{\frac 1{2m}},\ \ \ \ \Pro-\text{a.s.}
\end{equation}
where $\kappa(t)$ is a random variable, increasing with respect to $t$, and having finite moments of any order. 
\end{Lemma}

\begin{proof}
If we define $\rho(t):=D_xX(t,x)h-D_xX(t,y)h$, we have
\[\rho^\prime(t)=A\rho(t)+F^\prime(X(t,x))\rho(t)+\le[F^\prime(X(t,x))-F^\prime(X(t,y))\r]D_xX(t,y)h,\ \ \ \ \rho(0)=0,\]
and then,
\[\rho(t)=\int_0^t U^x_{t,s}\le[F^\prime(X(s,x))-F^\prime(X(s,y))\r]D_xX(s,y)h\,ds.\]
According to \eqref{6.2.9} and  \eqref{dp103}, this yields
\[\begin{array}{l}
\ds{|\rho(t)|_H\leq c\int_0^t\le|\le[F^\prime(X(s,x))-F^\prime(X(s,y))\r]D_xX(s,y)h\r|_H\,ds}\\
\vs
\ds{\leq c\int_0^t\le(1+|X(t,x)|_E^{2m-1}+|X(t,y)|_E^{2m-1}\r)\le|X(s,x)-X(s,y)\r|_E\le|D_xX(s,y)h\r|_H\,ds}\\
\vs
\ds{\leq \Lambda(t)\int_0^t\le(1+s^{-1+\frac 1{2m}}\r)\le|X(s,x)-X(s,y)\r|_E\le|D_xX(s,y)h\r|_H\,ds}\end{array}\]
and due to  \eqref{dp120} this allows to conclude.

\end{proof}

\begin{Remark}
{\em 
\begin{enumerate}
\item Since 
\[|U^x_{t,s}h|_E\leq c\,((t-s)\wedge 1)^{-\frac 14}|h|_H,\]
from the proof above, we easily see that for any $x, y \in\,E$ and $h \in\,H$
\begin{equation}
\label{dp230}
|D_xX(t,x)h-D_xX(t,y)h|_E\leq \kappa(t)\,|x-y|_E|h|_H (t\wedge 1)^{-\frac 14+\frac 1{2m}},\ \ \ \ \Pro-\text{a.s.}
\end{equation}
for some random variable $\kappa(t)$ as in Lemma \ref{l2.7}

\item Let $v(\cdot,x,h)$ be the process defined above  as 
\[\lim_{n\to \infty}D_xX(t,x_n)h_n,\ \ \ \text{in}\ H,\]
for any two sequences $\{x_n\}_{n \in\,\nat}$ and $\{h_n\}_{n \in\,\nat}$, converging in $H$ to $x$ and $h$, respectively.

Then, as above for $D_xX(t,x)h$, we have that for any $x,y,h \in\,H$ and $t>0$
\begin{equation}
\label{dp125bis}
|v(t,x,h)-v(t,y,h)|_H\leq \kappa(t)\,|x-y|_H|h|_H (t\wedge 1)^{-\frac 14+\frac 1{2m}},
\end{equation}
where $\kappa(t)$ is a random variable, increasing with respect to $t$, and having finite moments of any order. 

\end{enumerate}}
\end{Remark}

\section{The unperturbed semigroup}

In what follows, we shall denote by $P_t$  the Markov transition semigroup associated with equation \eqref{e1.1} in $E$. Namely
 \begin{equation}
\label{e1.5}
P_t\varphi(x)=\E\,\varphi(X(t,x)),\ \ \ \ x \in\,E, 
\end{equation}
for any $\varphi \in\,B_b(E)$ and $t\geq 0$, where $X(t,x)$ is the unique mild solution of equation \eqref{e1.1}. Moreover, we shall denote by $P^H_t$ the transition semigroup associated with equation \eqref{e1.1} in $H$. Namely
 \begin{equation}
\label{e1.5H}
P^H_t\varphi(x)=\E\,\varphi(X(t,x)),\ \ \ \ x \in\,H, 
\end{equation}
for any $\varphi \in\,B_b(H)$, where $X(t,x)$ is the unique generalized solution of equation \eqref{e1.1} in $H$. Notice that since $E$ is a Borel subset of $H$, if $x \in\,E$ and $\varphi \in\,B_b(E)$, then
\begin{equation}
\label{dp116}
P_t\varphi(x)=P^H_t\varphi(x),\ \ \ \ t\geq 0.
\end{equation}
For this reason,  in what follows we may not distinguish between $P_t$ and $P^H_t$ when it is not necessary.

In \cite[Theorem 6.5.1]{C01}) it is proved that the semigroup $P_t$ has a smoothing effect and, in spite of the polynomial growth of $f$, uniform bounds are satisfied by the derivatives of $P_t\varphi$. Actually, we have the following result.

\begin{Proposition}
\label{p2.3}
For any $\varphi \in\,B_b(E)$ and $t>0$, we have that $P_t\varphi \in\,C^3_b(E)$ and for any $0\leq i\leq j\leq 3$
\begin{equation}
\label{e2.3}
\|P_t\varphi\|_{C_b^j(E)}\leq c_j\,(t\wedge 1)^{-\frac{j-i}2}\|\varphi\|_{C^i_b(E)}.
\end{equation}
Moreover, it holds
\begin{equation}
\label{dp111}
\le<h,D(P_t\varphi)(x)\r>_E=\frac 1t\E\,\varphi(X(t,x))\int_0^t\le<D_xX(s,x)h,dw(s)\r>_H.
\end{equation}
\end{Proposition}

As a consequence of \eqref{dp111} and \eqref{dp120}, we have
\[\begin{array}{l}
\ds{\le|\le<h,D(P_t\varphi)(x)\r>_E\r|\leq \frac 1t\|\varphi\|_{C_b(E)}\le(\E\,\int_0^t|D_xX(s,x)h|_H^2\,ds\r)^{\frac 12}}\\
\vs
\ds{\leq c(t)(t\wedge 1)^{-\frac 12}\|\varphi\|_{C_b(E)}\,|h|_H,}
\end{array}\]
so that $D(P_t\varphi)(x)$ can be extended to a linear functional on $H$, for any $x \in\,E$ and $t>0$, and
\begin{equation}
\label{dp300}
|D(P_t\varphi)|_{C_b(E,H)}\leq c(t)(t\wedge 1)^{-\frac 12}\,\|\varphi\|_{C_b(E)}.
\end{equation}
In fact, the mapping $D(P_t\varphi):E\to H$ is Lipschitz-continuous, as shown in next lemma.

\begin{Lemma}
\label{l3.2}
For any $\varphi \in\,C_b(E)$, $x \in\,E$  and $t>0$ we have that $D(P_t\varphi)(x) \in\,H$ and for $j=0,1$
\begin{equation}
\label{dp235}
|D(P_t\varphi)(x)-D(P_t\varphi)(y)|_H\leq c(t)(t\wedge 1)^{-\frac {2-j}2}\|\varphi\|_{C_b^j(E)}|x-y|_E,\ \ \ \ x, y \in\,E.
\end{equation}
\end{Lemma}

\begin{proof}
Assume $\varphi \in\,C^1_b(E)$ and fix $x,y,h \in\,E$ and $t>0$. Then, 
\[\begin{array}{l}
\ds{\le<h,\le(D(P_t\varphi)(x)- D(P_t\varphi)(y)\r)\r>_E
=\frac 1t\E\,\le[\varphi(X(t,x))-\varphi(X(t,y))\r]\int_0^t \le<D_xX(s,x)h,dw(s)\r>_H}\\
\vs
\ds{+\frac 1t\E\,\varphi(X(t,y))\int_0^t \le<D_xX(s,x)h-D_xX(s,y)h,dw(s)\r>_H.}
\end{array}\]
Therefore, thanks to \eqref{dp120} and \eqref{dp125}, we get
\[\begin{array}{l}
\ds{\le|\le<h,\le(D(P_t\varphi)(x)- D(P_t\varphi)(y)\r)\r>_E\r| \leq \frac {c(t)}t\|\varphi\|_{C^1_b(E)}|x-y|_E\,\le(\E\int_0^t|D_xX(s,x)h|_H^2\,ds\r)^{\frac 12}}\\
\vs
\ds{+\frac 1t\,\|\varphi\|_{C_b(E)}\le(\E\int_0^t|D_xX(s,x)h-D_xX(s,y)h|_H^2\,ds\r)^{\frac 12}}\\
\vs
\ds{\leq c(t)(t\wedge 1)^{-\frac 12}\|\varphi\|_{C_b^1(E)}|h|_H|x-y|_E+c(t)(t\wedge 1)^{\frac 1{2m}}\|\varphi\|_{C_b(E)}|h|_H\,|x-y|_E,}
\end{array}\]
and this implies \eqref{dp235} for $j=1$. The case $j=0$ follows from \eqref{e2.3} and the semigroup law.

\end{proof}

Next, we recall that in \cite[Section 3]{CD12}, by using suitable interpolation estimates  for real-valued functions defined in the Banach space $E$, we have proved the following result.
\begin{Proposition}
\label{c2.5}
For any  $\theta\in (0,1)$ and $j=2,3$, there exists $c_{\theta,j}>0$ such that for all  $\varphi\in C^\theta_b(E)$  and all $t>0$ 
\begin{equation}
\label{e2.7}
\|P_t\varphi\|_{C^j_b(E)}\leq c_{\theta,j} (t\wedge 1)^{-\frac{j-\theta}2}\|\varphi\|_{C^\theta_b(E)}.
\end{equation}
\end{Proposition}

As a consequence of \eqref{dp235}, Proposition \ref{c2.5} and the semigroup law imply that for any $\varphi \in\, C^\theta_b(E)$, with $\theta \in\,[0,1]$, and for any  $x, y \in\,E$ and  $t>0$ 
\begin{equation}
\label{dp236}
|D(P_t\varphi)(x)-D(P_t\varphi)(y)|_H\leq c(t)(t\wedge 1)^{-\frac{2-\theta}2}\|\varphi\|_{C_b^\theta(E)}|x-y|_E,\ \ \ \ x, y \in\,E.
\end{equation}

\bigskip

In \cite[Theorem 7.3.1]{C01} we have also shown that for any $t>0$ the semigroup $P^H_t$ maps $B_b(H)$ into $C^1_b(H)$ and
\begin{equation}
\label{dp111H}
\le<h,D(P^H_t\varphi)(x)\r>_H=\frac 1t\E\,\varphi(X(t,x))\int_0^t\le<v(s,x,h),dw(s)\r>_H,
\end{equation}
where $v(\cdot,x,h)$ is the process defined in the previous section as the limit of the derivatives $D_xX(t,x_n)h_n$, where $\{x_n\}_{n \in\,\nat}$ and $\{h_n\}_{n \in\,\nat}$ are two sequences in $E$ converging respectively to $x$ and $h$ in $H$.
In particular, we have shown that
\begin{equation}
\label{dp113}
\|P^H_t\varphi\|_{C^1_b(H)}\leq c\,(t\wedge 1)^{-\frac{1}2}\|\varphi\|_{C_b(H)}.
\end{equation}
Thanks to \eqref{dp123}, we have that $P^H_t:C^\a_b(H)\to C^\a_b(H)$, for any $\a \in\,[0,1)$, and $P^H_t:\text{Lip}_b(H)\to \text{Lip}_b(H)$, with
\[\|P^H_t\varphi\|_{C^\a_b(H)}\leq c(t)\,\|\varphi\|_{C^\a_b(H)},\ \ \ \|P^H_t\varphi\|_{\text{Lip}_b(H)}\leq c(t)\,\|\varphi\|_{\text{Lip}_b(H)}.\]
Therefore, by interpolation, we have that 
$P^H_t:C^\a_b(H)\to C^\beta_b(H)$, for any $0\leq \a\leq \beta\leq 1$, and
\begin{equation}
\label{dp130}
\|P^H_t\varphi\|_{C^\beta_b(H)}\leq c(t)\,(t\wedge 1)^{-\frac{\beta-\a}2}\|\varphi\|_{C^\a_b(H)},\ \ \ t>0.
\end{equation}

\begin{Lemma}
\label{lemma3.2}
Let $0\leq \a<\beta<1$ and let $\varphi \in\,C^\a_b(H)$. Then $P^H_t\varphi \in\,C_b^{1+\beta}(H)$, for any $t>0$, and
\begin{equation}
\label{dp127}
\|P^H_t\varphi\|_{C^{1+\beta}_b(H)}\leq c(t)\le(t\wedge 1\r)^{-(\delta+\frac{\beta-\a}2)}\,\|\varphi\|_{C^{\a}_b(H)},
\end{equation}
where 
\begin{equation}
\label{dp150}
\d=\begin{cases}
\frac 12,  & \text{if}\ m=1,\\
\frac 34-\frac 1{2m},  &  \text{if}\ m>1.
\end{cases}
\end{equation} 

\end{Lemma}

\begin{proof}
Assume $\varphi \in\,C_b(H)$. Then, 
\[\begin{array}{l}
\ds{\le<h,D(P^H_t\varphi)(x)- D(P^H_t\varphi)(y)\r>_H 
=\frac 1t\E\,\le[\varphi(X(t,x))-\varphi(X(t,y))\r]\int_0^t \le<v(s,x,h),dw(s)\r>_H}\\
\vs
\ds{+\frac 1t\E\,\varphi(X(t,y))\int_0^t \le<v(s,x,h)-v(s,y,h),dw(s)\r>_H}
\end{array}\]
Therefore, thanks to \eqref{dp123} and \eqref{dp125bis}, we get
\[\begin{array}{l}
\ds{\le|\le<h,D(P^H_t\varphi)(x)- D(P^H_t\varphi)(y)\r>_H \r|\leq \frac {c(t)}t\|\varphi\|_{C_b^\beta(H)}\le(\E\int_0^t|v(s,x,h)|_H^2\,ds\r)^{\frac 12}|x-y|_H^\beta}\\
\vs
\ds{+\frac 1t\|\varphi\|_{C_b(H)}\le(\E\int_0^t|v(s,x,h)-v(s,y,h)|_H^2\,ds\r)^{\frac 12}}\\
\vs
\ds{\leq c(t)(t\wedge 1)^{-\frac 12}\|\varphi\|_{C_b^\beta(H)}|x-y|_H^\beta|h|_H+c(t)(t\wedge 1)^{-\frac 34+\frac 1{2m}}|h|_H.}

\end{array}\]
By the semigroup law, this implies that 
\[\le|D(P^H_t\varphi)(x)- D(P^H_t\varphi)(y)\r|_H\leq c(t)\,(t\wedge 1)^{-\delta}\|P^H_{t/2}\varphi\|_{C^\beta_b(H)}|x-y|_H^\beta,\]
so that \eqref{dp127} follows from \eqref{dp130}.

\end{proof}

\bigskip

By proceeding as in \cite{C94} (see also \cite[Appendix B]{C01}), we introduce the generator of $P_t$.
For any $\lambda>0$ and $\varphi\in C_b(E)$ we define
\begin{equation}
\label{e1.6}
F(\lambda)\varphi(x)=\int_0^\infty e^{-\lambda t}P_t\varphi(x)dt,\ \ \ \  x\in E.
\end{equation}

As proved e.g. in \cite[Proposition B.1.4]{C01}, there exists a unique $m$--dissipative operator $\mathcal L$ in $C_b(E)$ such that 
$$
R(\la,\L)=(\lambda-\mathcal L)^{-1}=F(\lambda),\ \ \ \lambda>0. 
$$
${\cal L}:D({\cal L})\subseteq C_b(E)\to C_b(E)$ is the {\em weak infinitesimal generator} of $P_t$. We would like to recall that, as proved in \cite{C94} (see also \cite{P99} and \cite{C01}), 
if $\varphi \in\,D(\L)$, then
\[\lim_{t\to 0}\Delta_t\varphi(x)=\L \varphi(x),\ \ \ x \in\,E,\]
and 
\[\sup_{t \in\,(0,1]}\|\Delta_t\varphi\|_{C_b(E)}<\infty,\]
where
\[\Delta_\e=\frac{1}{\e}\le(P_\e-I\r),\ \ \ \ \e \in\,(0,1].\]
Moreover, for any $\varphi \in\,D(\L)$ and $t\geq 0$, we have $P_t\varphi \in\,D(\L)$ and
\[\L P_t\varphi=P_t \L \varphi.\]
The mapping, $t\mapsto P_t\varphi(x)$ is differentiable, and
\[\frac{d}{dt}P_t\varphi(x)=\L P_t\varphi(x)=P_t\L\varphi(x),\ \ \ \ x \in\,E.\]

\bigskip

First of all, we notice that
\[\|R(\la,\L)\varphi\|_{C_b(E)}\leq \frac 1\la\,\|\varphi\|_{C_b(E)},\ \ \ \ \varphi \in\,C_b(E).\]
Moreover, as due to \eqref{e2.3} we have
\[\|P_t\varphi\|_{C_b^1(E)}\leq c\,(t\wedge 1)^{-\frac{1}2}\|\varphi\|_{C_b(E)},\]
we immediately have that $D(\L)\subset C^1_b(E)$ and
\begin{equation}
\label{der}
\sup_{x \in\,E}\|D(R(\la,\L)\varphi)(x)\|\leq \frac c{\sqrt{\la}}\|\varphi\|_{C_b(E)}.
\end{equation}
Notice that, as a consequence of \eqref{dp236}, if $\varphi \in\,C_b^\theta(E)$, with $\theta>0$, we have that $D(R(\la,\L)\varphi):E\to H$ is well defined, and
\begin{equation}
\label{dp237}
|D(R(\la,\L)\varphi)(x)-D(R(\la,\L)\varphi)(y)|\leq c\la^{-\frac \theta 2}\,\|\varphi\|_{C_b^\theta(E)},\ \ \ \ x, y \in\,E.
\end{equation}

\medskip

As for $P_t$ and $\L$, we can also introduce the weak generator $\L^H$ of the semigroup $P^H_t$. Due to 
\eqref{dp116}, for any $\la>0$ and $\varphi \in\,C_b(H)$ we have
\begin{equation}
\label{dp117}
R(\la,\L)\varphi(x)=R(\la,\L^H)\varphi(x),\ \ \ \ x \in\,E.
\end{equation}

Now, for any $\la>0$ and $\psi \in\,C_b(E)$, we consider the elliptic equation 
\begin{equation}
\label{elle}
\la \varphi-\L \varphi=\psi.
\end{equation}
As the resolvent set of $\L$ contains $(0,+\infty)$, we have that equation \eqref{elle} admits a unique solution in $C_b(E)$, which is given by $\varphi=R(\la,\L)\psi$.

In \cite{CD12} we have proved that in fact Schauder estimates are satisfied by the solution of equation \eqref{elle}.
\begin{Theorem}
\label{un3}
 Let $\psi\in C^\theta_b(E)$, with $\theta\in (0,1)$,     
  and let $\varphi=R(\la,\L)\psi$,  with $\lambda>0$.
 Then we have $ \varphi\in C^{2+\theta}_b(E)$ and there exists $c>0$ 
(independent  of $\psi$) such that
\begin{equation}
\label{e3.1}
 \|\varphi\|_{C^{2+\theta}_b(E)}\leq c\,\|\psi\|_{C^{\theta}_b(E)}.
\end{equation}
\end{Theorem}

Notice that, in view of Lemma \ref{lemma3.2}, we have 
\begin{equation}
\label{dp155}
\psi \in\,C^\a_b(H)\Longrightarrow \varphi=R(\la,\L)\psi \in\,C^{1+\beta}_b(H),
\end{equation}
for any $\beta<2(1-\d)+\a$, where $\d$ is the constant defined in \eqref{dp150}.

\medskip

 Next, we show that under a suitable condition on $\psi$ a {\em trace property} is satisfied by $D^2\varphi(x)$.
\begin{Theorem}
\label{un44}
For any $x \in\,E$    and  $\psi \in\,C^\theta_b(H)$, with $\theta >1/2$,   the series
\[\sum_{i=1}^\infty D^2\varphi(x)(e_i,e_i)\] is convergent and
  \begin{equation}
  \label{dp119}
\le|\sum_{i=1}^\infty D^2\varphi(x)(e_i,e_i)\r|\leq c\,\la^{\frac {\theta -2}4}\le(1+|x|_E^{2m-1}\r)\,\|\psi\|_{C^\theta_b(H)}.
\end{equation}
Moreover, the convergence is uniform for $x \in\,B_E(R)$, for any $R>0$.
\end{Theorem}

\begin{proof}
Assume first that $\psi \in\,C^1_b(H)$. If we differentiate in 
\[\le<h,D(P_t\psi)(x)\r>_E=\frac 1t\E\psi(X(t,x))\int_0^t\le<D_xX(s,x)h,dw(s)\r>_H,\]
along the direction $k \in\,E$, we get
\[\begin{array}{l}
\ds{D^2(P_t\psi)(x)(h,k)=\frac 1t\E\le<D_xX(t,x)k,D\psi(X(t,x))\r>_E\int_0^t\le<D_xX(s,x)h,dw(s)\r>_H}\\
\vs
\ds{+\frac 1t\E\,\psi(X(t,x))\int_0^t\le<D^2_xX(s,x)(h,k),dw(s)\r>_H.}
\end{array}\]
This means that for any $n, p \in\,\nat$
\[\begin{array}{l}
\ds{\sum_{i=n}^{n+p}D^2(P_t\psi)(x)(e_i,e_i)}\\
\vs
\ds{=\frac 1t\E\le<D_xX(t,x)\le(\sum_{i=n}^{n+p}\int_0^t\le<D_xX(s,x)e_i,dw(s)\r>_H e_i\r),D\psi(X(t,x))\r>_H}\\
\vs
\ds{+\frac 1t\,\E\psi(X(t,x))\int_0^t\le<\sum_{i=n}^{n+p}D^2_xX(s,x)(e_i,e_i),dw(s)\r>_H=:I_{1,p}^n(t)+I_{2,p}^n(t).}
\end{array}\]
Now, according to \eqref{dp120} and \eqref{dp300}, we have
\[\begin{array}{l}
\ds{|I_{1,p}^n(t)|\leq \frac ct \|\psi\|_{C_b^1(H)}\le(\E\,\le|\sum_{i=n}^{n+p}\int_0^t\le<D_xX(s,x)e_i,dw(s)\r>_H e_i\r|_H^2\r)^{\frac 12}}\\
\vs
\ds{=\frac ct \|\psi\|_{C_b^1(H)}\le(\E\,\int_0^t\sum_{i=n}^{n+p}\le|D_xX(s,x)e_i\r|^2_H \,ds\r)^{\frac 12}
}
\end{array}\]
and then, due to \eqref{dp100} we can conclude that for any $t>0$ and $p\geq 0$
\begin{equation}
\label{dp135}
\lim_{n\to\infty}I_{1,p}^n(t)=0.
\end{equation}
Moreover, for any $n \in\,\nat$
\begin{equation}
\label{dp136}
|I_{1,p}^n(t)|\leq c(t)\,(t\wedge 1)^{-\frac 34}\|\psi\|_{C^1_b(H)}.
\end{equation}

Next, according to \eqref{6.3.9} we have
\[\begin{array}{l}
\ds{|I_{2,p}^n(t)|\leq \frac 1t \|\psi\|_{C_b(H)}\le(\int_0^t\E\le|\sum_{i=n}^{n+p}D^2_xX(s,x)(e_i,e_i)\r|_H^2\,ds\r)^{\frac 12}}\\
\vs
\ds{\leq \frac{c(t)}t\|\psi\|_{C_b(H)}\le(\int_0^t\E\le(\sum_{i=n}^{n+p}\le|D^2_xX(s,x)(e_i,e_i)\r|_H\r)^2\,ds\r)^{\frac 12}.}
\end{array}\]
Then, as a consequence of \eqref{dp109}, we get
\begin{equation}
\label{dp137}
\lim_{n\to\infty}I_{2,p}^n(t)=0,
\end{equation}
and for any $n \in\,\nat$
\begin{equation}
\label{dp138}
|I_{2,p}^n(t)|\leq c(t)\,(t\wedge 1)^{-\frac 14}\le(1+|x|_E^{2m-1}\r)\,\|\psi\|_{C_b(H)}.
\end{equation}
Therefore, as $P_t\psi=P_{t/2}(P_{t/2}\psi)$ and $P_{t/2}\psi \in\,C^1_b(H)$, \eqref{dp135} and \eqref{dp137} imply that for any $p\geq 1$
\begin{equation}
\label{dp140}
\lim_{n\to\infty}\sum_{i=n}^{n+p}D^2(P_t\psi)(x)(e_i,e_i)=0.\end{equation}
Moreover, according to \eqref{dp130}, \eqref{dp136} and \eqref{dp138}, for any $n \in\,\nat$ we have
\begin{equation}
\label{dp141}
\le|\sum_{i=1}^{n}D^2(P_t\psi)(x)(e_i,e_i)\r|\leq c(t)(t\wedge 1)^{-\frac 34-\frac{1-\theta}2}\,\|\psi\|_{C^\theta_b(H)}.
\end{equation}

Now, as
\[\sum_{i=n}^{n+p}D^2\varphi(x)(e_i,e_i)=\int_0^{\infty}e^{-\la t}\sum_{i=n}^{n+p}D^2(P_t\psi)(x)(e_i,e_i)\,dt,\]
from \eqref{dp140} and \eqref{dp141} we can conclude that if $\theta>1/2$ then the series $\sum_{i=1}^{\infty}D^2\varphi(x)(e_i,e_i)$ is convergent and \eqref{dp119} holds.

The uniformity of the convergence with respect to $x \in\,B_R(E)$ is a consequence of the uniformity of the convergence in the series in Lemma \ref{l23} and Lemma \ref{dp110}.

\end{proof}

\begin{Remark}
\label{rem2.7}
{\em In view of Remark \ref{rem2.6}, if $J_n=nR(n,A)$, then we immediately have that the series
\[\sum_{i=1}^\infty D^2\varphi(J_n x)(J_n e_i,J_n e_i),\]
is uniformly convergent, with respect to $x \in\,B_R(E)$ and $n \in\,\nat$.}
\end{Remark}

\section{The vectorial unperturbed  semigroup}
\label{sec4}

If $\Phi \in\,C^j_b(E,E)$, for some positive integer $j$, we have
$D^i\Phi(x)(f_1,\ldots,f_i) \in\,E$, for any $x_1,\ldots,x_i \in\,E$ and any integer $i\leq j$. Moreover, if for  $v \in\,E^\star$ we denote 
\[\varphi_v(x):=\langle \Phi(x),v\rangle_E,\ \ \ x \in\,E,\]
we have that $\varphi_v \in\,C^j_b(E)$, and
\begin{equation}
\label{un14}
\langle D^i\Phi(x)(x_1,\ldots,x_i),v\rangle_E=D^i\varphi_v(x)(x_1,\ldots,x_i),
\end{equation}
so that
\begin{equation}
\label{un15}
\|D^i\Phi(x)\|_{{\cal L}^i(E,E)}= \sup_{|v|_{E^\star}\leq 1}|D^i\varphi_v(x)|_{{\cal L}^i(E)},\ \ \ x \in\,E.
\end{equation}

Now, for any $\Phi \in\,C_b(E,E)$, we define
\[\widehat{P}_t\Phi(x)=\E\,\Phi(X(t,x)),\ \ \ x \in\,E,\ \ \ \ t\geq0.\]

Clearly ${\widehat P}_t$ maps $C_b(E,E)$ into itself and for any $v \in\,E^\star$
\begin{equation}
\label{un12}
\langle{\widehat P}_t\Phi(x),v\rangle_E=P_t\varphi_v(x),\ \ \ x \in\,E,\,\ \ \ t\geq0.
\end{equation}
Moreover, it is possible to adapt the arguments used in \cite[Theorem 6.5.1]{C01} and prove that for any $t>0$
\[{\widehat P}_t:C_b(E,E)\to C^3_b(E,E).\]
This implies the following result.
\begin{Proposition}
\label{un16}
For any  $0\leq i\leq j\leq 3$ and $t>0$
\begin{equation}
\label{vece2.7}
\|\widehat{P}_t\Phi\|_{C^j_b(E,E)}\leq c\, (t\wedge 1)^{-\frac{j-i}2}\|\Phi\|_{C^i_b(E,E)}.
\end{equation}
\end{Proposition}
\begin{proof}
According to \eqref{un15} and \eqref{un12}, we have
\[\begin{array}{l}
\ds{\sup_{x \in\,E}\,|D^j(\widehat{P}_t\Phi)(x)|_{\L^j(E)} =\sup_{|v|_{E^\star}\leq 1}\sup_{x \in\,E}\,|D^j(\langle \widehat{P}_t\Phi,v\rangle_E)(x)|_{\L^j(E,\reals)}}\\
\vs
\ds{=\sup_{|v|_{E^\star}\leq 1}\sup_{x \in\,E}\,|D^j(P_t\varphi_v)(x)|_{\L^j(E,\reals)},}
\end{array}\]
Then, as 
\[\sup_{x \in\,E}\,|D^i\varphi_v(x)|_{\L^j(E,\reals)}\leq |v|_{E^\star}\sup_{x \in\,E}\,|D^i\Phi(x)|_{\L^j(E)},\]
by using \eqref{e2.3} we can conclude.

\end{proof}
Next, as 
\[\|\varphi_v\|_{C^\theta_b(E)}\leq |v|_{E^\star}\|\Phi\|_{C^\theta_b(E,E)},\]
by proceeding as we did in \cite{CD12} by using interpolation, we obtain the following  generalization of Proposition \ref{c2.5} to the vectorial case.

\begin{Proposition}
\label{cvec2.5}
For any  $\theta\in (0,1)$ and $j=2,3$, there exists $c_{\theta,j}>0$ such that for all  $\Phi\in C^\theta_b(E,E)$  and all $t>0$ 
\begin{equation}
\label{evec2.7}
\|\widehat{P}_t\Phi\|_{C^j_b(E,E)}\leq c_{\theta,j} (t\wedge 1)^{-\frac{j-\theta}2}\|\Phi\|_{C^\theta_b(E,E)}.
\end{equation}
\end{Proposition}

Notice that, due to \eqref{vece2.7}, by proceeding as in Lemma \ref{l3.2}, we have that $D(\widehat{P}_t\Phi)(x) \in\,\L(H)$, for any $\Phi \in\,C_b(E,E)$, $x \in\,E$ and $t>0$, and, thanks to \eqref{evec2.7}, as in \eqref{dp236} we have that
\begin{equation}
\label{dp240}
|D(\widehat{P}_t\Phi)(x)-D(\widehat{P}_t\Phi)(y)|_{\L(H)}\leq c(t)\,(t\wedge 1)^{-\frac{2-\theta}2}\|\Phi\|_{C^\theta_b(E,E)}|x-y|_E,\ \ \ \ x,y \in\,E,
\end{equation}
for any $\Phi \in\,C^\theta_b(E,E)$.

Now, as in the case of $P_t$, we can define the infinitesimal generator  of ${\widehat P}_t$, as the unique $m$-dissipative  operator $\widehat{{\cal L}}:D(\widehat{{\cal L}})\subset C_b(E,E)\to C_b(E,E)$ such that
\[R(\la,\widehat{\L})=(\la-\widehat{{\cal L}})^{-1}=\widehat{F}(\la),\ \ \ \la>0,\]
where
\[{\widehat F}(\la)\Phi(x)=\int_0^\infty e^{-\la t}\widehat{P}_t\Phi(x)\,dt,\ \ \ x \in\,E.\]
Due to 
\eqref{un12}, it is immediate to check that
$\Phi \in\,D(\widehat{{\cal L}})$ if and only if $\langle \Phi,v\rangle_E \in\,D({\cal L})$, for any $v \in\,E^\star$, and
\begin{equation}
\label{un13}
\langle \widehat{{\cal L}}\Phi,v\rangle_E={\cal L}\,\varphi_v.
\end{equation}

As for ${\cal L}$, we have that 
\[\|R(\la,\widehat{\L})\Phi\|_{C_b(E,E)}\leq \frac 1\la\,\|\Phi\|_{C_b(E,E)},\ \ \ \ \Phi \in\,C_b(E,E).\]
As a consequence of \eqref{vece2.7},
\begin{equation}
\label{vecder}
\sup_{x \in\,E}|D(R(\la,\widehat{\L})\Phi)(x)|_{\L(E)}\leq \frac c{\sqrt{\la}}\|\Phi\|_{C_b(E,E)},
\end{equation}
and, from \eqref{dp240}, as in \eqref{dp237}, if  $\Phi \in\,C^\theta_b(E,E)$ we get
\begin{equation}
\label{dp241}
|D(R(\la,\widehat{\L})\Phi)(x)-D(R(\la,\widehat{\L})\Phi)(y)|_{\L(H)}\leq \frac{c}{\la^{-\frac \theta 2}}\|\Phi\|_{C^\theta_b(E,E)}\,|x-y|_E,\ \ \ \ x, y \in\,E.
\end{equation}

Moreover,  as a consequence of Proposition \ref{cvec2.5}, we have
\begin{Theorem}
\label{vecun3}
 Let $\Psi\in C^\theta_b(E,E)$, with $\theta\in (0,1)$,     
  and let $\Phi=R(\lambda,\widehat{\mathcal L})\Psi$,  with $\lambda>0$.
 Then we have $ \Phi\in C^{2+\theta}_b(E,E)$ and there exists $c>0$ 
(independent  of $\Psi$) such that
\begin{equation}
\label{vece3.1}
 \|\Phi\|_{C^{2+\theta}_b(E,E)}\leq c\,\|\Psi\|_{C^{\theta}_b(E,E)}.
\end{equation}
\end{Theorem} 

Finally, we would like to stress that, in view of \eqref{un13}, if $\Phi$ solves the equation
\[\la\,\Phi-\widehat{\L}\,\Phi=\Psi,\]
then for any $v \in\,E^\star$ the function $\varphi_v$ solves the equation
\[\la\varphi_v-\L\varphi_v=\psi_v.\]

\medskip

Now, let $\Phi \in\,C_b(E_{\e_1},E_\e)$, for some $\e_1\leq \e_0$ and $\e>0$. According to \eqref{dp161}, we have that $\Phi(X(t,x)) \in\,E_\e$, for any $t>0$ and $x \in\,E$. Therefore, as the mapping $x \in\,E\mapsto X(t,x) \in\,E_{\e_1}$ is continuous, we have that 
\[\Phi \in\,C_b(E_{\e_1},E_\e)\Longrightarrow \widehat{P}_t\Phi \in\,C_b(E,E_\e),\ \ \ \ t>0.\]
In fact, we have the following smoothing property
\begin{Lemma}
\label{l100}
If $\Phi \in\,C_b(E,E)\cap B_b(E_{\e_1},E_\e)$, for some $\e>0$ and $\e_1\leq \e_0$, then $\widehat{P}_t\Phi:E\to E_\e$ is differentiable and 
\begin{equation}
\label{dp164}
\sup_{x \in\,E}|D(\widehat{P}_t\Phi)(x)|_{\L(E,E_\e)}\leq c(t)\,(t\wedge 1)^{-\frac 12}\,\|\Phi\|_{B_b(E_{\e_1},E_\e)}.
\end{equation}
Therefore
\begin{equation}
\label{dp163}
\sup_{x \in\,E}|D(R(\la,\widehat{\L})\Phi)(x)|_{\L(E,E_\e)}\leq c\,\la^{-\frac 12}\,\|\Phi\|_{B_b(E_{\e_1},E_\e)}.
\end{equation}
\end{Lemma}

\begin{proof}
As $\Phi \in\,C_b(E,E)$, we have $\widehat{P}_t\Phi \in\,C^1_b(E,E)$ and for any $x,h \in\,E$ 
\[
D(\widehat{P}_t\Phi)(x)\cdot h=\frac 1t\E\,\Phi(X(t,x))\int_0^t\langle D_xX(s,x)h,dw(s)\rangle_H.\]
Thanks to \eqref{dp120}, as $\Phi \in\,B_b(E_{\e_1},E_\e)$, we get
\[|D(\widehat{P}_t\Phi)(x)\cdot h|_{E_\e}\leq \frac{c(t)}{t}\|\Phi\|_{B_b(E_{\e_1},E_\e)}\sqrt{t}|h|_E.\]
This implies  \eqref{dp164} and hence \eqref{dp163} .

\end{proof}

\medskip

Next, we introduce the vectorial semigroup in $H$, by defining 
\[\widehat{P}^H_t\Phi(x)=\E\,\Phi(X(t,x)),\ \ \ \ x \in\,H,\ \ \ \ t\geq 0,\]
for any $\Phi \in\,C_b(H,H)$, where $X(t,x)$ is the unique generalized solution of \eqref{e1.1} in $H$. $\widehat{\L}^H$ is the corresponding  weak generator, defined as $\widehat{\L}$.

By arguing as in the proof of Proposition \ref{un16}, from \eqref{dp130} for any $0\leq \a\leq \beta\leq 1$ we have  
\begin{equation}
\label{dp1300}
\|\widehat{P}^H_t\Phi\|_{C^\beta_b(H,H)}\leq c(t)\,(t\wedge 1)^{-\frac{\beta-\a}2}\|\Phi\|_{C^\a_b(H,H)},\ \ \ t>0.
\end{equation}
and
from Lemma \ref{lemma3.2} we have that $\widehat{P}^H_t$ maps $C_b^\a(H,H)$ into $C_b^{1+\beta}(H,H)$, for any $0\leq \a\leq \beta<1$, and 
\begin{equation}
\label{dp1270}
\|\widehat{P}^H_t\Phi\|_{C^{1+\beta}_b(H,H)}\leq c(t)\le(t\wedge 1\r)^{-(\delta+\frac{\beta-\a}2)}\,\|\Phi\|_{C^{\a}_b(H,H)},
\end{equation}
where $\d$ is the constant defined in \eqref{dp150}.

Finally, from Theorem \ref{un44}, we get that if $\Psi \in\,C_b^\theta(H,H)\cap C^\a_b(E,E)$, with $\a>0$ and $\theta >1/2$, then the series $\sum_{i=1}^\infty D^2(R(\la,\widehat{\L})\Psi)(x)(e_i,e_i)$ is convergent in $H$, uniformly with respect to $x \in\,B_R(E)$, and 
\begin{equation}
\label{dp157}
\le|\sum_{i=1}^\infty D^2(R(\la,\widehat{\L})\Psi)(x)(e_i,e_i)\r|_H\leq c\,\la^{\frac \theta 2-\frac 14}\le(1+|x|_E^{2m-1}\r)\,\|\Psi\|_{C^\theta_b(H,H)}.
\end{equation}

\section{Perturbations}
We study now suitable perturbations of the Kolmogorov operator $\L$, obtained by adding a first order term. We distinguish the case the drift if regular and then in particular there is uniqueness for the corresponding stochastic equation, and the case the drift is only H\"older continuous.

\subsection{Regular perturbations}

We are here concerned with the  operator
 \begin{equation}
 \label{e2.1}
\widehat{ \mathcal L}\Phi+D\Phi \cdot B, \ \ \ \ \Phi \in\,D(\widehat{\L}),
 \end{equation} 
where  $B\in C^1_b(E,E)$.

We consider the stochastic differential equation
\begin{equation}
\label{e2.2}
\left\{\begin{array}{l}
dY(t)=[AY(t)+F(Y(t))+B(Y(t))]\,dt + dw(t),\\
\\
Y(0)=x\in E,
\end{array}\right.
\end{equation} 
which we write in the following mild form
  \begin{equation}
\label{e2.3a}
\begin{array}{lll}
Y(t)&=&\ds e^{tA}x+\int_0^te^{(t-s)A}F(Y(s))ds+W_A(t)+\int_0^te^{(t-s)A}B(Y(s))ds.
\end{array}
\end{equation}
By reasoning as in \cite[Proposition 6.2.2]{C01} for equation \eqref{e1.4},  equation \eqref{e2.2} has a unique mild solution $Y(t,x) \in\,L^p(\Omega;C([0,T];E))$, for any $p\geq 1$ and $T>0$. 

Next lemma shows that a stochastic non-linear variation of constants formula holds, which allows to write equation 
\eqref{e2.2} in terms of the solution of equation \eqref{e1.3} and of the associated first derivative equation. The proof, that we omit, follows from the same argument used in \cite{B66}, adapted to this stochastic case.

\begin{Lemma}
\label{Lvcf}
Let $Y(t,x)$ and $X(t,x)$ be the solutions of equations \eqref{e2.2} and \eqref{e1.3}, respectively. Then we have
\begin{equation}
\label{vcf}
Y(t,x)=X(t,x)+\int_0^t U^{Y(s,x)}_{t,s}B(Y(s,x))\,ds,
\end{equation}
where  $U^x_{t,s}h$ is the solution of the first derivative equation
\[\eta^\prime(t)=A\eta(t)+F^\prime(X(t,x))\eta(t),\ \ \ \ \eta(s)=h,\]
for any $x \in\,E$ and $0\leq s\leq t$ (see \eqref{random} and \eqref{dp225} and Lemma \ref{l2.2}).

\end{Lemma}

Now, we define the corresponding transition semigroup
 \begin{equation}
 \label{E2.3}
 \widehat{Q}_t\Phi(x)=\E[\Phi(Y(t,x))],\quad \Phi\in C_b(E,E),
 \end{equation} 
whose infinitesimal generator $\widehat{\mathcal N}$  is defined in the same way we did before for the generator $\widehat{\L}$ of the semigroup $\widehat{P}_t$. This means that $\widehat{\N}$ is the $m$--dissipative operator in $C_b(E,E)$, whose domain  $D(\widehat{\mathcal N})$ is characterized as the linear space of all functions $\Phi\in C_b(E,E)$ such that  there exists the limit
\[\lim_{t\to 0}\frac{\widehat{Q}_t\Phi(x)-\Phi(x)}{t}= \widehat{\mathcal N} \Phi(x),\ \ \ x  \in\,E,\]
and 
\[\sup_{t \in\,(0,1]}\sup_{x \in\,E} \frac 1t\le|\widehat{Q}_t\Phi(x)-\Phi(x)\r|_{E}<\infty.\]
Notice that, as we are assuming $B \in\,C^2_b(E,E)$, the same arguments used for equation \eqref{e1.3} and the semigroup $\widehat{P}_t$ adapt to equation \eqref{e2.2} 
and hence we have
\[\sup_{x \in\,E}|D(\widehat{Q}_t\Phi)(x)|_{\L(E)}\leq c\,(t\wedge 1)^{-\frac 12}\|\Phi\|_{C_b(E,E)}.\]
This implies that
$D(\widehat{\N})\subset C^1_b(E,E)$ and
\begin{equation}
\label{derN}
\sup_{x \in\,E}|D((\la-\widehat{\N})^{-1}\Phi)(x)|_{\L(E)}\leq \frac c{\sqrt{\la}}\|\Phi\|_{C_b(E,E)}.
\end{equation}

\begin{Proposition}
\label{p2.1}
We have  $D(\widehat{\mathcal N})=D(\widehat{\mathcal L})$
and 
\begin{equation}
\label{e2.4}
\widehat{\mathcal N}\Phi=\widehat{\mathcal L}\Phi+D\Phi\cdot B,\ \ \ \ \Phi\in D(\widehat{\mathcal L})=D(\widehat{\N}). 
\end{equation}
\end{Proposition} 

\begin{proof}
In view of Lemma \ref{Lvcf} and of the fact that $D(\widehat{\L})\subset C^1_b(E,E)$, we have for any $\Phi \in\,D(\widehat{\L})$
\[\begin{array}{l}
\ds{\Phi(X(t,x))=\Phi(Y(t,x))-\le[\Phi(X(t,x)+R(t,x))-\Phi(X(t,x))\r]}\\
\vs
\ds{=\Phi(Y(t,x))-\int_0^1D\Phi(X(t,x)+\theta R(t,x))\,d\theta\cdot R(t,x),}
\end{array}\]
where
\[R(t,x)=\int_0^t U_{t,s}^{Y(s,x)}B(Y(s,x))\,ds.\]
so that
\[\frac{\widehat{Q}_t\Phi(x)-\Phi(x)}t=\frac{\widehat{P}_t\Phi(x)-\Phi(x)}t+\E\,\int_0^1D\Phi(X(t,x)+\theta R(t,x))\,d\theta\cdot \frac 1t\,R(t,x).\]
Now, for any $x \in\,E$ we have
\[\begin{array}{l}
\ds{\lim_{t\to 0}\E\,\int_0^1D\Phi(X(t,x)+\theta R(t,x))\,d\theta\cdot \frac 1t\,R(t,x)=D\Phi(x)\cdot B,}
\end{array}\]
and, due to \eqref{dp103}, for $t \in\,(0,1]$ we have
\[\le|\int_0^1D\Phi(X(t,x)+\theta R(t,x))\,d\theta\cdot \frac 1t\,R(t,x)\r|\leq c\,\|B\|_{C_b(E,E)}\,\|\Phi\|_{C^1_b(E,E)}.\]
As we are assuming that $\Phi \in\,D(\widehat{\L})$, this allows us to conclude that
\[\lim_{t\to 0}\frac{\widehat{Q}_t\Phi(x)-\Phi(x)}t=\le(\L \Phi+D\Phi\cdot B\r)(x),\ \ \ \ x \in\,E,\]
and hence $\Phi \in\,D(\widehat{\N})$ and \eqref{e2.4} holds.

The inclusion $D(\widehat{\N})\subset D(\widehat{\L})$ follows from an analogous argument, as $D(\widehat{\N})\subset C^1_b(E,E)$.

\end{proof}

\subsection{H\"older perturbations} 

Now, we aim to study the elliptic equation
\begin{equation}
\label{e3.19}
\lambda\Phi(x)-\widehat{\mathcal L}\Phi(x)-D\Phi(x)\cdot B(x)=G(x),\ \ \ \ x \in\,E,
\end{equation}
where $\lambda>0$,  $G\in C^\alpha_b(E,E)$ and  $B\in C^\alpha_b(E;E)$,  for some $\alpha\in(0,1).$
We are going to show   the following result. 
\begin{Theorem}
\label{t3.11}
Let $B\in C^\alpha_b(E;E)$  and $G\in C^\alpha_b(E,E)$, for some $\alpha\in(0,1).$
Then, for any $\la>0$  there exists a unique solution $\Phi\in D(\widehat{\mathcal L})\cap C^{2+\alpha}_b(E,E)$ of equation \eqref{e3.19}.  Moreover, for any $\e \in\,[0,2]$
  there exists $c_\epsilon>0$ (independent of $\lambda$ and $G$) such that
 \begin{equation}
 \label{e2.7c}
 \|\Phi\|_{C^{2+\alpha-\epsilon}_b(E,E)}\leq c_\e\,\le(\frac {1}{\la^{\e/2}}+ \frac{1}{\lambda}\r)\; \|G\|_{C^\alpha_b(E,E)}.
 \end{equation} 
\end{Theorem} 

\begin{proof}

 {\it Step 1}. Let  $\Phi\in D(\widehat{\mathcal N})\cap C^{2+\alpha}_b(E,E)$ be a solution of  \eqref{e3.19}. Then we have
\begin{equation}
\label{e3.20}
\|\Phi\|_{C_b(E,E)}\leq \frac1\lambda\;\|G\|_{C_b(E,E)}.
\end{equation} \medskip

 By an approximation result due to Valentine \cite{V45}, we can choose a sequence $\{B_n\}\subset C^1_b(E,E)$  uniformly convergent to $B$. Then, thanks to Proposition \ref{p2.1} we can  write equation \eqref{e3.19} as
\begin{equation}
\label{e3.21}
\lambda\Phi-\widehat{\mathcal L}\Phi -D\varphi\cdot B_n=G_n,
\end{equation}
where
\begin{equation}
\label{e3.22}
G_n(x)=G(x) +D\Phi\rangle_E(x)\cdot (B(x)-B_n(x)),\ \ \ x \in\,E.
\end{equation}
Consider now the stochastic differential equation
\begin{equation}
\label{e3.23}
\left\{\begin{array}{l}
dX_n(t)=(AX_n(t)+F(X_n(t))+B_n(X_n(t)))dt + dw(t),\\
\\
X_n(0)=x\in H,
\end{array}\right.
\end{equation} 
which has a unique solution $X_n(t,x)$.
Then, if we  introduce  the transition semigroup
 \begin{equation}
 \label{e3.24}
\widehat{ Q}^n_t\Phi(x)=\E\,\Phi(X_n(t,x)),\quad \Phi\in C_b(E,E),
 \end{equation} 
 and the corresponding generator $\widehat{\N}_n$, we have
$$
\lambda \Phi-\widehat{\mathcal N}_n \Phi=G_n.
$$
Consequently,
$$
\| \Phi\|_{C_b(E,E)}\leq \frac1\lambda\;\|G_n\|_{C_b(E,E)}.
$$
Now the conclusion follows letting $n\to\infty$.\medskip

{\it Step 2}.  There exists a constant $c>0$ such that if  $\Phi\in D(\widehat{\mathcal L})\cap C^{2+\alpha}_b(E,E)$ is a solution of  \eqref{e3.19} then
\begin{equation}
\label{e3.25}
\|\Phi\|_{C^{2+\alpha}_b(E,E)}\leq c\,\|G\|_{C^\alpha_b(E,E)}.
\end{equation} \medskip

By \eqref{e3.19} and Schauder's estimate \eqref{e3.1}, there exists $c>0$ (independent of $\la$ and $f$) such that
\[
\|\Phi\|_{C^{2+\alpha}_b(E,E)}\leq c\,(\|G\|_{C^\alpha_b(E,E)}+ \|B\|_{C^\alpha_b(E,E)} \|\Phi\|_{C^{1+\a}_b(E,E)}).
\]
Now the conclusion follows from standard interpolatory estimates,
as, by \eqref{e3.20}
\[\begin{array}{l}
\ds{\|\Phi\|_{C^{2+\alpha}_b(E,E)}\leq c\,\le(\|G\|_{C^\alpha_b(E,E)}+ \|B\|_{C^\alpha_b(E,E)}\,\|\Phi\|_{C^{2+\a}_b(E,E)}^{\frac{1+\a}{2+\a}}\|\Phi\|_{C_b(E,E)}^{\frac{1}{2+\a}}\r)}\\
\vs
\ds{\leq c^\prime\,\|G\|_{C^\alpha_b(E,E)}+\frac 12\|\Phi\|_{C^{2+\a}_b(E,E)}.}
\end{array}\]

\medskip

{\it Step 3}. For any $\e\geq 0$, let us consider the equation
\begin{equation}
\label{un6}
\la \Phi-\widehat{\L}\Phi-\e D\Phi\cdot B=G.
\end{equation}
Then, the set $
\Lambda:=\{\epsilon\in [0,1]:\; \eqref{un6} \;\mbox{has a unique solution $\Phi\in D(\widehat{\mathcal L})\cap C_b^{2+\alpha}(E)$}\} 
$
is open.

\medskip

Assume $\e_0 \in\,\Lambda$. We want to prove that for $\e$ sufficiently close to $\e_0$  equation 
\eqref{un6} has a unique solution. If we set
\begin{equation}
\label{un7}
\la \Phi-\widehat{\L}\Phi-\e_0 D\Phi\cdot B=\Psi,
\end{equation}
equation   \eqref{un6} becomes
\begin{equation}
\label{e3.19a}
\Psi- T_{\la,\e}\Psi=G,
\end{equation}
where
\begin{equation}
\label{e3.20a}
 T_{\la,\e}\Psi(x)=(\e-\e_0)DR(\lambda,\mathcal L)\Psi\cdot B.
\end{equation}
According to \eqref{der}, we  have
$$
\begin{array}{l}
\ds{\|T_{\la,\e}\Psi\|_{C_b(E,E)}\leq |\e_0-\e|\,\|B\|_{C_b(E,E)}\; \sup_{x \in\,E}|D(R(\lambda,\widehat{\mathcal L})\Psi)(x)\|_{\L(E)}}\\
\vs
\ds{\leq \frac {|\e_0-\e|}{\sqrt{\la}}\,\|B\|_{C_b(E,E)}\;\|\Psi\|_{C_b(E,E)}.}
\end{array}
$$
and
$$
\begin{array}{l}
\ds{[T_{\la,\e}\Psi]_{C^\alpha_b(E,E)}\leq  |\e_0-\e|\|B\|_{C_b(E,E)}\; [DR(\lambda,\widehat{\mathcal L})\Psi]_{C^\alpha_b(E,E)}}\\
\vs
\ds{+|\e_0-\e|\,
\|B\|_{C^\alpha_b(E,E)}\; \|D(R(\lambda,\widehat{\mathcal L})\Psi)\|_{C_b(E,E)}\leq c\frac {|\e_0-\e|}{\sqrt{\la}}\,\|B\|_{C^\alpha_b(E,E)}\|\Psi\|_{C^\alpha_b(E,E)}.}
\end{array}
$$
Consequently
\begin{equation}
\label{e3.21a}
\|T_{\la,\e}\Psi\|_{C^\alpha_b(E,E)}\leq \frac{2 c\,|\e_0-\e|}{\sqrt{\la}}\|B\|_{C^\alpha_b(E,E)}\|\Psi\|_{C^\alpha_b(E,E)}
\end{equation}
so that $T_{\la,\e}$ is a contraction on $C_b^{\alpha}(E)$ provided
$|\e_0-\e|<\sqrt{\la}/2c\,\|B\|_{C^\alpha_b(E,E)}.$
By the contraction principle, this allows to conclude that there exists $\Psi \in\,C^\a_b(E,E)$ solving \eqref{e3.19a} and, as $\e_0 \in\,\Lambda$, this implies that there exists  a unique solution $\Phi$  for equation \eqref{e3.19}, which belongs to $C^{2+\a}_b(E,E)$.

\medskip

{\it Step 4}. Conclusion.\medskip

We  use the continuity method. The set $\Lambda$ introduced above is non empty, as $0 \in\,\Lambda$. Moreover, due to the previous step, it is open. Therefore, if  we show that is closed, we have $\Lambda=[0,1]$ and the conclusion follows. Let $\epsilon_n\to \bar{\epsilon}$ with $(\epsilon_n)\subset \Lambda$. We have
$$
\lambda (\Phi_{\e_n}-\Phi_{\e_m})-\widehat{\mathcal L}(\Phi_{\e_n}-\Phi_{\e_m})-\epsilon_nD(\Phi_{\e_n}-\Phi_{\e_m})\cdot B=(\epsilon_n-\epsilon_m)D\Phi_{\epsilon_m}\cdot B.
$$
From the Schauder estimate \eqref{e3.1} and \eqref{e3.25}, we get 
\[\begin{array}{l}
\ds{\|\Phi_{\e_n}-\Phi_{\e_m}\|_{C^{2+\a}_b(E,E)}\leq c\,|\epsilon_n-\epsilon_m|\,\|B\|_{C^\a_b(E,E)}\|\Phi_{\e_m}\|_{C^{1+\a}_b(E,E)}}\\
\vs
\ds{\leq c\,|\epsilon_n-\epsilon_m|\,\|B\|_{C^\a_b(E,E)}\,\|G\|_{C^\a_b(E,E)},}
\end{array}\]
and then we conclude that $\{\Phi_{\epsilon_n}\}_{n \in\,\nat}$ is a Cauchy sequence in $C^{2+\a}_b(E,E)$. This implies that the sequence $\{\Phi_{\e_n}\}_{n \in\,\nat}$ converges to some $\bar{\Phi} \in\,D(\widehat{\mathcal L})\cap C^{2+\a}_b(E,E)$ and such $\bar{\Phi}$ is the unique solution of  equation \eqref{un6}, for $\bar{\e}$. 

\medskip

{\it Step 5}. Proof of estimate \eqref{e2.7c}.

\medskip

Due to \eqref{evec2.7}, we have
\[\|R(\la,\widehat{\L})\Phi\|_{C^{2+\a-\e}_b(E,E)}\leq c\int_0^\infty e^{-\la t}(t\wedge 1)^{-1+\frac \e 2}\,dt\,\|G\|_{C^\a_b(E,E)},\]
so that \eqref{e2.7c} follows immediately.
\end{proof}

In fact, the solution $\Phi$ of equation \eqref{e3.19} satisfies the following properties.

\begin{Lemma}
\label{l54}
Assume that $B, G \in\,C_b^\a(E,E)$, for some $\a>0$. Then, if $\Phi$ is the solution of equation \eqref{e3.19}, if $\la$ is large enough then $D\Phi(x) \in\,\L(H)$, for any $x \in\,E$, and
\begin{equation}
\label{dp242}
|D\Phi(x)-D\Phi(y)|_{\L(H)}\leq c\,|x-y|_E,\ \ \ \ x,y \in\,E.
\end{equation}

Moreover, if we also assume that $G \in\,B_b(E_{\e_1},E_\e)$, for some $\e>0$ and $\e_1\leq \e_0$, then 
\begin{equation}
\label{dp168}
|\Phi(x)-\Phi(y)|_{E_\e}\leq  c\,\la^{-\frac 12}\|G\|_{C_b(E_\e,E_\e)}|x-y|_E,\ \ \ \ x,y \in\,E.
\end{equation}
\end{Lemma}

\begin{proof}
By proceeding as in Step 3 in the proof of Theorem \ref{t3.11}, for $\la$ large enough the mapping 
\[T_\la:C_b^\a(E,E)\to C_b^\a(E,E),\ \ \ \ \Psi\mapsto T_\la \Psi=D(R(\la,\widehat{\L})\Psi)\cdot B,\]
is a contraction. Therefore, as $\Phi=R(\la,\widehat{\L})(I-T_\la)^{-1}G$, due to \eqref{dp241} we have that $D\Phi(x) \in\,\L(H)$, for any $x \in\,E$, and \eqref{dp242} holds.

In view of Lemma \ref{l100} and \eqref{dp163}, we have that for any $\Psi \in\,C_b(E,E)\cap B_b(E_{\e_1},E_\e)$ the mapping $x \in\,E\mapsto D(R(\la,\widehat{\L})\Psi)(x)\cdot B(x) \in\,E_\e$ is well defined and continuous,  and 
\[\sup_{x \in\,E}\,|D(R(\la,\widehat{\L})\Psi)(x)\cdot B(x)|_{E_\e}\leq c\,\la^{-\frac 12}\|\Psi\|_{B_b(E_{\e_1},E_\e)}.\]
This implies that if $\la$ is large enough
\[T_\la:C_b(E,E)\cap B_b(E_{\e_1},E_\e)\to C_b(E,E)\cap B_b(E_{\e_1},E_\e),\]
is a contraction. Therefore, as $\Phi=R(\la,\widehat{\L})(I-T_\la)^{-1}G$, due to \eqref{dp163} we have that $\Phi$ is continuous from $E$ into $E_\e$.

Now, for any $x, y \in\,E$, we have
\[\Phi(x)-\Phi(y)=\int_0^1D\Phi(\theta x+(1-\theta)y)\cdot(x-y)\,d\theta,\]
and then, as
\[D\Phi(\theta x+(1-\theta)y)\cdot(x-y)=D(R(\la,\widehat{\L})(I-T_\la)^{-1}G)(\theta x+(1-\theta)y)\cdot (x-y),\]
according to \eqref{dp163} we conclude
\[|\Phi(x)-\Phi(y)|_{E_\e}\leq c\,\la^{-\frac 12}\|(I-T_\la)^{-1}G\|_{B_b(E_{\e_1},E_\e)}|x-y|_E.\]

\end{proof}

Finally, we show that under stronger assumptions on $B$ and $G$, the solution $\Phi$ of equation \eqref{e3.19} has some further 
properties.

\begin{Theorem}
\label{theo5.5}
Assume that $B, G \in\,C^\a_b(E,E)\cap C^\theta_b(H,H)$, for some $\theta \in\,[0,1)$, and take $\la$  sufficiently large. Then,
\begin{enumerate}
\item we have $\Phi \in\,C^{1+\theta}_b(H,H)$;
\item if  $\theta>1/2$,   the series $\sum_{i=1}^\infty D^2\Phi(x)(e_i,e_i)$ and  $\sum_{i=1}^\infty D^2\Phi(J_nx)(J_n e_i,J_n e_i)$ are convergent in $H$, uniformly with respect to $n \in\,\nat$ and $x \in\,B_R(E)$, for any $R>0$.  In particular, for any $x \in\,E$
\begin{equation}
\label{dp160}
\lim_{n\to\infty}\sum_{i=1}^\infty D^2\Phi(J_nx)(J_n e_i,J_n e_i)=\sum_{i=1}^\infty D^2\Phi(x)(e_i,e_i)\ \ \ \text{in}\ E.
\end{equation}
\end{enumerate}

\begin{proof}
{\it Proof of 1.} According to \eqref{dp1270} we have that $\widehat{P}_t$ maps $C^\theta_b(H,H)$ into $C_b^{1+\theta}(H,H)$ with
\[\|\widehat{P}_t\Psi\|_{C^{1+\theta}_b(H,H)}\leq c(t)(t\wedge 1)^{-\d}\|\Psi\|_{C^\theta_b(H,H)},\]
where $\d$ is the constant, strictly less than $1$, defined in \eqref{dp150}.
This implies that
\begin{equation}
\label{dp1550}
R(\la,\widehat{\L}):C^\theta_b(H,H)\to C^{1+\theta}_b(H,H),
\end{equation}
and
\[\sup_{x \in\,H}|D(R(\la,\widehat{\L})\Psi)(x)h|_H\leq \frac{c}{\la^{\d-1}}\,|h|_H\|\Psi\|_{C_b(H,H)},\]
Hence, as we are assuming $B \in\,C_b^\theta(H,H)$, if we pick $\la$ large enough, we have that the mapping
\[T_\la:C^\theta_b(H,H)\to C^\theta_b(H,H),\ \ \ \ \Psi\mapsto T_\la \Psi=D(R(\la,\widehat{\L})\Psi)\cdot B,\]
is a contraction.
Now, as $\Phi=R(\la,\widehat{\L})(I-T_\la)^{-1}G$ and $G \in\,C_b^\theta(H,H)$, thanks to \eqref{dp1550}, we conclude that $\Phi \in\,C^{1+\theta}_b(H,H)$.

\medskip

{\it Proof of 2.} Due to the previous step, we have  $D\Phi\cdot B+G \in\,C_b^\theta(H,H)$. Then, as we have
\[\Phi=R(\la,\widehat{\L})\le[D\Phi\cdot B+G\r],\]
and we are assuming $\theta>1/2$, we can conclude from \eqref{dp157} and from Remark \ref{rem2.7}.

\end{proof}

\end{Theorem}

\section{Pathwise uniqueness}

We want to prove that pathwise uniqueness holds in the class of mild solutions of the equation
\begin{equation}
\label{perturbe1.3}
\left\{\begin{array}{lll}
dY(t)=[A Y(t)+F(Y(t))+B(Y(t))]dt+dw(t),\\
\\
Y(0)=x,
\end{array}\right.
\end{equation}
where $A$, $F$ and $W$ are as in section \ref{sec2} and $B$ satisfies the following condition.

\begin{Hypothesis}
\label{H2}
There exist $\a, \e>0$ and $\e_1\leq \e_0$ such that 
\[B \in\,C_b^\a(E,E)\cap B_b(E_{\e_1},E_\e).\]

\end{Hypothesis}

\begin{Remark}
{\em We have already seen that the mappings $B$ described in Subsection \ref{subsection examples} are both in $C^\a_b(E,E)$. 
Moreover, they belong to $B_b(E_{\e_1},E_\e)$, for suitable positive constants as in Hypothesis \ref{H2}.

Let \[
B\left(  x\right)  \left(  \xi\right)  =b\left(  x\left(  \xi_{0}\right)
\right)  g\left(  \xi\right),\ \ \ \ \xi \in\,[0,1],
\]
for some $g\in E$, $\xi_{0}\in\left[  0,1\right]  $ and  $b\in C_{b}^{\alpha}\left(
\mathbb{R},\mathbb{R}\right) $. If we assume that $g \in\,E_\e$, then $B$ maps $E_\e$ into $E_\e$.
Actually,
for any $\xi_1, \xi_2 \in\,[0,1]$ we have
\[B(x)(\xi_1)-B(x)(\xi_2)=b(x(\xi_0))(g(\xi_1)-g(\xi_2)),\]
so that
\[[B(x)]_{E_\e}\leq \|b\|_\infty\,[g]_{E_\e}.\]

\medskip

Now, let 
\[
B\left(  x\right)  \left(  \xi\right)  =b\left(  \max_{s\in\left[
0,\xi\right]  }x\left(  s\right)  \right),\ \ \ \ \xi \in\,[0,1],
\]
for some $b\in C_{b}^{\alpha}\left(
\mathbb{R},\mathbb{R}\right) $. Then,
$B$ maps $E_\e$ into $E_{\e\a}$. Actually, for any $\xi_1, \xi_2 \in\,[0,1]$, with $\xi_1>\xi_2$,  we have
\[\begin{array}{l}
\ds{|B(x)(\xi_1)-B(x)(\xi_2)|\leq [b]_{C^\a(\reals)}|\max_{s\leq \xi_1}x(s)-\max_{s\leq \xi_2}x(s)|^\a}\\
\vs
\ds{=[b]_{C^\a(\reals)}(x(\bar{\xi}_1)-x(\bar{\xi}_2))^\a,}
\end{array}\]
for some $\bar{\xi}_i\leq \xi_i$, $i=1,2$. If $\bar{\xi}_1\leq \xi_2$, then $x(\bar{\xi}_1)=x(\bar{\xi}_2)$ and we are done. Thus, assume $\xi_2\leq \bar{\xi}_1\leq \xi_1$. We have 
\[\begin{array}{l}
\ds{0\leq x(\bar{\xi}_1)-x(\bar{\xi}_2)\leq x(\bar{\xi}_1)-x(\xi_2)\leq |x(\bar{\xi}_1)-x(\xi_2)|\leq |\bar{\xi}_1-\xi_2|^\e\leq |\xi_1-\xi_2|^\e.}
\end{array}\]

}
\end{Remark}

\bigskip

As in \cite{DF10}, the main idea here is to represent the {\em bad term} $B(Y(t))$ in terms of nicer objects, by using It\^o's formula.

To this purpose, we show how we can point-wise approximate the mapping $B$ by nicer mappings $B_m$.

\begin{Lemma}
\label{l5.2}
Under Hypothesis  \ref{H2}, there exists a sequence
$\{B_m\}_{m \in\,\nat}\subset C^\a_b(E,E)\cap C^\infty_b(H,E)$ such that
\[\le\{\begin{array}{l}
\ds{\lim_{m\to \infty}|B_m(x)-B(x)|_E=0,\ \ \ \ x \in\,E,}\\
\vs
\ds{\sup_{m \in\,\nat}\|B_m\|_{C^\a_b(E,E)}<\infty.}
\end{array}\r.\]
\end{Lemma}

\begin{proof}
For any $m \in\,\nat$ we define
\[T_m \xi=\sum_{k=1}^m\xi_k e_k,\ \ \ \xi \in\,\reals^m,\ \ \ \ \ Q_mx=(x_1,\ldots,x_m),\ \ \ P_m x=\sum_{k=1}^m x_ke_k,\ \ \ x \in\,H,\]
where $x_k=\le<x,e_k\r>_H$. If we define
\[\hat{P}_mx =\frac 1m\sum_{k=1}^m P_k x,\ \ \ \ x \in\,H,\]
then
Fej\'er's Theorem states that $\hat{P}_mx$ converges to $x$ in $E$, as $m\uparrow \infty$, when $x \in\,E$. In particular, as a consequence of the uniform boundedness theorem, 
\begin{equation}
\label{fine1}
\sup_{m \in\,\nat}\|\hat{P}_m\|_{\L(E)}<\infty.
\end{equation}
Now, as for any $x \in\,H$ we have $\hat{P}_m x \in\,E$, we can define
\[B_m(x)=\int_{\reals^m}B(\hat{P}_m(x-T_m \xi))\rho_{m}(\xi)\,d\xi,\ \ \ \ x \in\,H,\]
where $\rho_m \in\,C_c^\infty(\reals^m)$ is a probability density with support in $\{ \xi \in\,\reals^m,\ |\xi|_{\reals^m}\leq 1/m^2\}$. 
We have
clearly that $B_m:H\to E$ and due to \eqref{fine1} for any $x,y \in\,E$
\[\begin{array}{l}
\ds{|B_m(x)-B_m(y)|_E\leq \int_{\reals^m}|B(\hat{P}_m(x-T_m \xi))-B(\hat{P}_m(y-T_m \xi))|_E\,\rho_{m}(\xi)\,d\xi}\\
\vs
\ds{\leq [B]_{C^\a(E,E)}|\hat{P}_mx-\hat{P}_my|_E^\a\int_{\reals^m}\rho_{m}(\xi)\,d\xi\leq c\,|x-y|^\a_E.}
\end{array}\]
This implies that $\{B_m\}_{m \in\,\nat}$ is a bounded sequence in $C^\a(E,E)$. 

Moreover, as $\hat{P}_{m_1} P_{m_2}=\hat{P}_{m_1}$, for any $m_1\leq m_2$, with a change of variable we have
\[B_m(x)=\int_{\reals^m}B(\hat{P}_m(x-T_m \xi))\,\rho_{m}(\xi)\,d\xi=\int_{\reals^m}B(\hat{P}_m T_m \eta)\,\rho_{m}(\eta+Q_m x)\,d\eta,\]
and, as $\rho_m$ is in $C_c^\infty(\reals^m)$, this implies that $B_m \in\,C^\infty_b(H,E)$.

Finally, for any $x \in\,E$ we have
\[\begin{array}{l}
\ds{|B_m(x)-B(x)|_E\leq c_\a\,[B]_{C^\a(E,E)}\int_{\reals^m}\le(|\hat{P}_mx-x|_E^\a+|\hat{P}_m T_m \xi|_E^\a\r)\,\rho_{m}(\xi)\,d\xi,}
\end{array}\]
and then,
as for any $|\xi|_{\reals^m}\leq 1/m^2$ we have
\[|\hat{P}_m T_m \xi|_E\leq \frac 1m\sum_{k\leq m}\sum_{i\leq k}|\xi_i|\leq \frac {m(m+1)}{2 m^3}\to 0,\ \ \ \ \text{as}\ m\to \infty,\]
recalling that $|\hat{P}_mx-x|_E\to 0$, we conclude that $B_m(x)$ converges to $B(x)$ in $E$, for any $x \in\,E$.

\end{proof}

Now we define 
\[Y_n(t,x)=J_n Y(t,x),\]
where $J_n=nR(n,A)$, we have
\begin{equation}
\label{un20}
\le\{\begin{array}{l}
\ds{dY_n(t)=\le[A Y_n(t)+J_n F(Y(t))+J_n B(Y(t))\r]\,dt+J_n dw(t),}\\
\vs
\ds{Y_n(0)=J_n x.}
\end{array}
\r.
\end{equation}

Notice that, if $Y(t,x)$ is a mild solution  of equation \eqref{perturbe1.3}, we have that $Y_n(t,x)$ is a strong solution of equation \eqref{un20}, that is
\begin{equation}
\label{un25}
Y_n(t,x)=J_n x+\int_0^t \le[AY_n(s,x)+J_n F(Y(s,x))+J_n B(Y(s,x))\r]\,ds+J_nW(t).
\end{equation}
Moreover,
\begin{equation}
\label{dp170}
\begin{cases}
\ds{|Y_n(t,x)|_E\leq |Y(t,x)|_E,}\\
\vs
\ds{\lim_{n\to \infty}|Y_n(t,x)-Y(t,x)|_E=0,}
\end{cases}\end{equation}
for any $t\geq 0$ and $x \in\,E$, $\Pro$-a.s.

Now, for each $\la>0$ we consider the elliptic equation 
\begin{equation}
\label{un24}
\la \Phi_m-\widehat{\L}\Phi_m-D\Phi_m\cdot B_m=B_m,
\end{equation}
where $B_m$ is the  mapping introduced in Lemma \ref{l5.2}. Later on we will choose $\la>0$ large enough. We denote by $\Phi_m$ its unique solution. According to what we have seen in Theorem \ref{t3.11}, we have that $\Phi_m \in\,C^{2+\a}_b(E,E)$
and as the sequence $\{B_m\}_{m \in\,\nat}$ is equi-bounded in $C^\a_b(E,E)$, we have 
\begin{equation}
\label{dp169}
\sup_{m \in\,\nat}\|\Phi_m\|_{C^{2+\a}_b(E,E)}<\infty.
\end{equation} 

\begin{Lemma}
\label{lemma-dp200}
If $\la$ is large enough, we have
\begin{equation}
\label{dp205}
\lim_{m\to \infty}\,|\Phi_m(x)-\Phi(x)|_{E}=0,\ \ \ \ x \in\,E,
\end{equation}
and
\begin{equation}
\label{dp210}
\lim_{m\to \infty}|D\Phi(x)-D\Phi_m(x)|_{\L(H,E)}=0,\ \ \ \ x \in\,E.
\end{equation}
\end{Lemma}

\begin{proof}
We have
\[\Phi-\Phi_m=R(\la,\widehat{\L})\le(\Psi-\Psi_m\r),\]
where
\[\Psi=(I-T_\la)^{-1}B,\ \ \ \ \Psi_m=(I-T_{\la,m})^{-1}B_m,\]
and
\[T_{\la}\Psi(x)=D(R(\la,\widehat{\L})\Psi)(x)\cdot B(x),\ \ \ \ T_{\la,m}\Psi(x)=D(R(\la,\widehat{\L})\Psi)(x)\cdot B_m(x),\ \ \ \ x \in\,E.\]
Therefore, if we show that the sequence $\{\Psi_m\}_{m \in\,\nat}$ is bounded in $C_b(E,E)$ and
\begin{equation}
\label{fine2}
\lim_{m\to \infty}|\Psi(x)-\Psi_m(x)|_E=0,\end{equation}
in view of what we have seen in Section \ref{sec4}, it is immediate to check  that 
\begin{equation}
 \label{fime3}
 \lim_{m\to\infty}|R(\la,\widehat{\L})(\Psi-\Psi_m)(x)|_E=0,\ \ \ \ x \in\,E,
\end{equation}
 and 
\begin{equation}
 \label{fime4}
 \lim_{m\to\infty}|D(R(\la,\widehat{\L})(\Psi-\Psi_m))(x)|_{\L(H,E)}=0,\ \ \ x \in\,E,
\end{equation}
 and \eqref{dp205} and \eqref{dp210} follow.
 
 We have
\[\begin{array}{l}
\ds{(\Psi-\Psi_m)-T_{\la}(\Psi-\Psi_m)=\le[\Psi-T_{\la}\Psi\r]-\le[\Psi_m-T_{\la,m}\Psi_m\r]+
D(R(\la,\widehat{\L})\Psi_m)\cdot(B-B_m)}\\
\vs
\ds{=\le[I+D(R(\la,\widehat{\L})\Psi_m)\r]\cdot(B_m-B).}
\end{array}\]
If $\la>0$ is large enough, the mapping $T_{\la}:C_b(E,E)\to C_b(E,E)$ is a contraction and then
\[\begin{array}{l}
\ds{\Psi-\Psi_m=(I-T_{\la})^{-1}\le[I+D(R(\la,\widehat{\L})\Psi_m)\r]\cdot(B_m-B).}
\end{array}\]
Now, due to \eqref{vecder}, for any $x \in\,E$ we have
\[|D(R(\la,\widehat{\L})\Psi_m)\cdot(B_m-B)(x)|_E\leq \frac{c}{\sqrt{\la}}\|\Psi_m\|_{C_b(E,E)}|B_m(x)-B(x)|_E,\]
so that
\[\lim_{m\to\infty}\le| \le[I+D(R(\la,\widehat{\L})\Psi_m)\r]\cdot(B_m-B)(x)\r|_E=0,\ \ \ \ x \in\,E,\]
and
\[\sup_{m \in\,\nat}\,\le\|\le[I+D(R(\la,\widehat{\L})\Psi_m)\r]\cdot(B_m-B)\r\|_{C_b(E,E)}<\infty.\]
According to \eqref{fime4}, this implies
\[\lim_{m\to \infty}|T_\la \le[I+D(R(\la,\widehat{\L})\Psi_m)\r]\cdot(B_m-B)(x) |_E=0,\ \ \ \ x \in\,E.\]
Therefore, as
\[\begin{array}{l}
\ds{(I-T_\la)^{-1}\le[I+D(R(\la,\widehat{\L})\Psi_m)\r]\cdot(B_m-B)}\\
\vs
\ds{=\sum_{k=0}^\infty T_\la^k\le[I+D(R(\la,\widehat{\L})\Psi_m)\r]\cdot(B_m-B),}
\end{array}\]
and $T_\la$ is a contraction, we conclude that
\[\lim_{m\to\infty}|(I-T_\la)^{-1}\le[I+D(R(\la,\widehat{\L})\Psi_m)\r]\cdot(B_m-B)(x)|_E=0,\ \ \ \ x \in\,E,\]
so that \eqref{fine2} follows.

\end{proof}

As $\Phi_m$ belongs to $C^2_b(E,E)$, and
 $X_n(t,x)$ solves equation \eqref{un25}, we can use the generalization of the Ito formula in the space of continuous functions, proved in \cite[Appendix A]{CD12bis}, and we have
 \[\begin{array}{l}
 \ds{d\Phi_m(Y_n(t,x))=\le[\frac 12\sum_{i=1}^\infty D^2\Phi_m(Y_n(t,x))(e_i,e_i)+D\Phi_m\cdot B_m(Y_n(t,x))\r]\,dt}\\
 \vs
 \ds{+D\Phi_m(Y_n(t,x))\cdot J_n dw(t)+R_{n,m}(t)\,dt,}
 \end{array}\]
 where
 \[\begin{array}{l}
 \ds{R_{n,m}(t)=D\Phi_m(Y_n(t,x))\cdot\le[J_n F(Y(t,x))-F(Y_n(t,x))\r]}\\
 \vs
 \ds{+D\Phi_m(Y_n(t,x))\cdot\le[J_n B(Y(t,x))-B_m(Y_n(t,x))\r]}\\
 \vs
 \ds{+\frac 12\sum_{i=1}^\infty D^2\Phi_m(Y_n(t,x))(J_ne_i,J_ne_i)-\frac 12\sum_{i=1}^\infty D^2\Phi_m(Y_n(t,x))(e_i,e_i).}
 \end{array}\]
 Therefore, since $\Phi_m$ solves equation \eqref{un24}, we have
\[\begin{array}{l}
\ds{ d\Phi_m(Y_n(t,x))=\la\,\Phi_m(Y_n(t,x))\,dt-B_m(Y_n(t,x))\,dt}\\
\vs
\ds{+D\Phi_m(Y_n(t,x))\cdot J_n dw(t)+R_{n,m}(t)\,dt,}
\end{array}\]
and then, for any $0\leq s\leq t$ we have
\[\begin{array}{l}
\ds{d\,e^{(t-s)A}\Phi_m(Y_n(s,x))=(\la-A)e^{(t-s)A}\Phi_m(Y_n(s,x))\,ds-e^{(t-s)A}B_m(Y_n(s,x))\,ds}\\
\vs
\ds{+e^{(t-s)A}D\Phi_m(Y_n(s,x))\cdot J_n dw(s)+e^{(t-s)A}R_{n,m}(s)\,ds.}
\end{array}\]
This yields
\begin{equation}
\label{dp171}
\begin{array}{l}
\ds{\int_0^te^{(t-s)A}B_m(Y_n(s,x))\,ds=e^{tA}\Phi_m(J_nx)-\Phi_m(Y_n(t,x))}\\
\vs
\ds{+\int_0^t (\la-A)e^{(t-s)A}\Phi_m(Y_n(s,x))\,ds+\int_0^te^{(t-s)A}D\Phi_m(Y_n(s,x))\cdot J_n dw(s)}\\
\vs
\ds{+
\int_0^te^{(t-s)A}R_{n,m}(s)\,ds.}
\end{array}\end{equation}

\begin{Lemma}
\label{lemma6.2}
Under Hypotheses  \ref{H1} and \ref{H2}, if $Y(t,x)$ is a mild solution of \eqref{perturbe1.3}, we
have
\begin{equation}
\label{dp215}
\begin{array}{l}
\ds{\int_0^te^{(t-s)A}B(Y(s,x))\,ds=e^{tA}\Phi(x)-\Phi(Y(t,x))}\\
\vs
\ds{+\int_0^t (\la-A)e^{(t-s)A}\Phi(Y(s,x))\,ds+\int_0^te^{(t-s)A}D\Phi(Y(s,x))\cdot dw(s),}
\end{array}\end{equation}
where $\Phi$ is the unique solution of the elliptic equation
\begin{equation}
\label{dp220}
\la \Phi-\widehat{\L}\Phi-D\Phi\cdot B=B.
\end{equation}
\end{Lemma}

\begin{proof}
We first take the limit in \eqref{dp171} as $n$ goes to infinity and then the limit as $m$ goes to infinity.

\bigskip

{\it Step 1: Limit as $n$ goes to infinity}

\medskip

Due to a maximal regularity result, we have
\[\begin{array}{l}
\ds{\int_0^T\le|\int_0^t (\la-A)e^{(t-s)A}\Phi_m(Y_n(s,x))\,ds-\int_0^t (\la-A)e^{(t-s)A}\Phi_m(Y(s,x))\,ds\r|_H^2\,dt}\\
\vs
\ds{\leq c_\la(T)\int_0^T\le|\Phi_m(Y_n(s,x))-\Phi_m(Y(s,x))\r|_H^2\,ds,}
\end{array}\]
for some constant $c_\la(T)$, independent of $\Phi_m$ and $Y(t,x)$.
Then, according to \eqref{dp170} and \eqref{dp169}, we conclude 
\begin{equation}
\label{dp200}
\lim_{n\to \infty}\int_0^{\cdot} (\la-A)e^{(\cdot-s)A}\Phi_m(Y_n(s,x))\,ds=\int_0^{\cdot} (\la-A)e^{(\cdot-s)A}\Phi_m(Y(s,x))\,ds.
\end{equation}
$\Pro$-a.s. in $L^2(0,T;H)$.

Next, as $\Phi_m \in\,C^{1+\theta}_b(H,H)$, due to Theorem \ref{theo5.5} we have
\[\begin{array}{l}
\ds{\E\le|\int_0^te^{(t-s)A}D\Phi_m(Y_n(s,x))\cdot J_n dw(s)-\int_0^te^{(t-s)A}D\Phi_m(Y(s,x))\cdot J_n dw(s)\r|_H^2}\\
\vs
\ds{=\E\int_0^t\sum_{i=1}^\infty\le|e^{(t-s)A}\le[D\Phi_m(Y_n(s,x))-D\Phi_m(Y(s,x))\r]J_n e_i\r|_H^2\,ds}\\
\vs
\ds{\leq \E\int_0^t\sum_{i=1}^\infty \sum_{j=1}^\infty\le|\le<\le[D\Phi_m(Y_n(s,x))-D\Phi_m(Y(s,x))\r]e_i,e^{(t-s)A}e_j\r>_H\r|^2}\\
\vs
\ds{=\E\int_0^t\sum_{j=1}^\infty e^{-2(t-s)\a_j}\sum_{i=1}^\infty\le|\le<e_i,\le[D\Phi_m(Y_n(s,x))-D\Phi_m(Y(s,x))\r]^\star e_j\r>_H\r|^2}\\
\vs
\ds{\leq c(t)\int_0^t(t-s)^{-\frac 12}\E\,|Y_n(s,x)-Y(s,x)|_H^2\,ds.}

\end{array}\]
As clearly
\[\E\le|\int_0^te^{(t-s)A}D\Phi_m(Y(s,x))\cdot J_n dw(s)-\int_0^te^{(t-s)A}D\Phi_m(Y(s,x))\cdot dw(s)\r|_H^2=0,\]
due to \eqref{dp170} this implies that for any $t \in\,[0,T]$
\begin{equation}
\label{dp201}
\lim_{n\to\infty}\int_0^te^{(t-s)A}D\Phi_m(Y_n(s,x))\cdot J_n dw(s)=\int_0^te^{(t-s)A}D\Phi_m(Y(s,x))\cdot dw(s),
\end{equation}
in $L^2(\Omega;H)$.

Finally, as according to \eqref{dp160} we immediately have
\[\lim_{n\to\infty}R_{n,m}(t)=D\Phi_m(Y(t,x))\cdot\le[B(Y(t,x))-B_m(Y(t,x))\r],\]
from the dominated convergence theorem we have that for any $t \in\,[0,T]$
\begin{equation}
\label{dp202}
\lim_{n\to\infty}\int_0^te^{(t-s)A}R_{n,m}(s)\,ds=\int_0^te^{(t-s)A}D\Phi_m(Y(s,x))\cdot\le[B(Y(s,x))-B_m(Y(s,x))\r]\,ds,
\end{equation}
$\Pro$-a.s. in $H$.

Therefore, collecting together \eqref{dp200}, \eqref{dp201} and \eqref{dp202}, from \eqref{dp169} and \eqref{dp171} we get
\[\begin{array}{l}
\ds{\lim_{n\to \infty}\int_0^{t}e^{(t-s)A}B_m(Y_n(s,x))\,ds=e^{tA}\Phi_m(x)-\Phi_m(Y(t,x))}\\
\vs
\ds{+\int_0^t (\la-A)e^{(t-s)A}\Phi_m(Y(s,x))\,ds+\int_0^te^{(t-s)A}D\Phi_m(Y(s,x))\cdot dw(s)}\\
\vs
\ds{+\int_0^te^{(t-s)A}D\Phi_m(Y(s,x))\cdot\le[B(Y(s,x))-B_m(Y(s,x))\r]\,ds,}
\end{array}\]
in $L^2(0,T;L^2(\Omega;H))$. 

\bigskip

{\it Step 2: Limit as $m$ goes to infinity}

\medskip
By using arguments analogous to those used in the previous step, from Lemma \ref{lemma-dp200} we have
\[\begin{array}{l}
\ds{\int_0^{t}e^{(t-s)A}B(Y(s,x))\,ds=\lim_{m\to \infty}\,\lim_{n\to \infty}\int_0^{t}e^{(t-s)A}B_m(Y_n(s,x))\,ds=e^{tA}\Phi(x)-\Phi(Y(t,x))}\\
\vs
\ds{+\int_0^t (\la-A)e^{(t-s)A}\Phi(Y(s,x))\,ds+\int_0^te^{(t-s)A}D\Phi(Y(s,x))\cdot dw(s),}
\end{array}\]
in $L^2(0,T;L^2(\Omega;H))$. This allows to conclude that \eqref{dp215} holds.

\end{proof}

The previous lemma has provided a nice representation of the {\em bad term}
\[\int_0^t e^{(t-s)A}B(Y(s,x))\,ds,\]
for any mild solution $Y(t,x)$ of equation \eqref{perturbe1.3} and this allows to conclude that pathwise uniqueness holds.

\begin{Theorem}
Under Hypotheses \ref{H1} and \ref{H2}, pathwise uniqueness holds for equation \eqref{perturbe1.3}.
\end{Theorem}

\begin{proof}

Let $Y_1(t,x)$ and $Y_2(t,x)$ be two mild solutions of equation \eqref{perturbe1.3} in $L^2(\Omega;C([0,T];E))$.  Thanks to Lemma \ref{lemma6.2} we have
\[\begin{array}{l}
\ds{Y_1(t,x)-Y_2(t,x)=\Phi(Y_2(t,x))-\Phi(Y_1(t,x))}\\
\vs
\ds{+\int_0^t e^{(t-s)A}\le[F(Y_1(s,x))-F(Y_2(s,x))\r]\,ds}\\
\vs
\ds{
+\int_0^t (\la-A)e^{(t-s)A}\le[\Phi(Y_1(s,x))-\Phi(Y_2(s,x))\r]\,ds}\\
\vs
\ds{+
\int_0^te^{(t-s)A}\le[D\Phi(Y_1(s,x))-D\Phi(Y_2(s,x))\r]\cdot dw(s)=:\sum_{i=1}^4 I_i(t),}\end{array}\]
where $\Phi$ is the unique solution of the elliptic  equation \eqref{dp220}.

Now, for any $R>0$ we denote by $\tau_R$ the stopping time 
\[\tau_R:=\inf \le\{\,t\geq 0,\ |Y_1(t,x)|_E\vee |Y_2(t,x)|_E\geq R\r\},\]
and we define
\[\tau:=\lim_{R\to \infty}\tau_R.\]
Clearly, we have
$\Pro(\tau=T)=1$.

\medskip

{\it Step 1.} According to \eqref{e2.7c}, with $\e=1+\a$, we have
\[|\Phi(x)-\Phi(y)|_E\leq c(\la)|x-y|_E,\]
for some function $c(\la)\downarrow 0$, as $\la\uparrow \infty$. Therefore, for any $p\geq 1$ there exists $\la_p>0$ such that
\begin{equation}
\label{dp221}
|I_1(t)|_E^p\leq \frac 12|Y_1(t,x)-Y_2(t,x)|^p_E,\ \ \ \ t\geq 0,\ \ \ \la\geq \la_p.
\end{equation}

\medskip

{\it Step2.} 
As $F$ is locally Lipschitz-continuous in $E$, for any $R>0$ we have
\[\begin{array}{l}
\ds{|I_2(t\wedge \tau_R)|_E\leq c\int_0^{t\wedge \tau_R}\le(1+|Y_1(s,x)|_E^{2m}+|Y_2(s,x)|_E^{2m}\r)|Y_1(s,x)-Y_2(s,x)|_E\,ds}\\
\vs
\ds{\leq  c\,(1+R^{2m})\int_0^t|Y_1(s\wedge \tau_R,x)-Y_2(s\wedge \tau_R,x)|_E\,ds,}
\end{array}\]
so that, for any $p\geq 1$
\begin{equation}
\label{dp256}
|I_2(t\wedge \tau_R)|_E^p\leq c_{R,p}(t)\int_0^t|Y_1(s\wedge \tau_R,x)-Y_2(s\wedge \tau_R,x)|_E^p\,ds.
\end{equation}

\medskip

{\it Step 3.} By a factorization argument, for any $R>0$ and $\beta \in\,(0,1)$ we have
\[I_3(t\wedge \tau_R)=c_\beta\int_0^{t\wedge \tau_R}(t\wedge \tau_R-s)^{\beta-1}(\la-A)e^{(t\wedge \tau_R-s)A}v_\beta(s)\,ds,\]
where
\[v_\beta(s)=\int_0^s (s-\si)^{-\beta}(\la-A)e^{(s-\si)A}\le[\Phi(Y_1(\si,x))-\Phi(Y_2(\si,x))\r]\,d\si.\]
This implies that for any $p>1/\beta$
\[\begin{array}{l}
\ds{|I_3(t\wedge \tau_R)|^p_E\leq c_{\beta,p}\le(\int_0^{t\wedge \tau_R}(t\wedge \tau_R-s)^{(\beta-1)\frac p{p-1}}\r)^{p-1}\int_0^{t\wedge \tau_R}I_{\{s<\tau_R\}}|v_\beta(s)|^p_E\,ds}\\
\vs
\ds{\leq c_{\beta,p}(t)\int_0^tI_{\{s<\tau_R\}}|v_\beta(s)|^p_E\,ds.}
\end{array}\]
Now, in view of \eqref{dp168}, if we assume $\beta<\e$ and $p>1/\beta\vee 1/(\e-\beta)$, we have
\[\begin{array}{l}
\ds{\int_0^tI_{\{s<\tau_R\}}|v_\beta(s)|^p_E\,ds}\\
\vs
\ds{\leq \int_0^tI_{\{s<\tau_R\}}\le|\int_0^s (s-\si)^{-\beta}(\la-A)e^{(s-\si)A}\le[\Phi(Y_1(\si,x))-\Phi(Y_2(\si,x))\r]\,d\si\r|^p\,ds}\\
\vs
\ds{\leq c_p(\la)\int_0^tI_{\{s<\tau_R\}}\le(\int_0^s (s-\si)^{-\beta-1+\e}\le|\Phi(Y_1(\si,x))-\Phi(Y_2(\si,x))\r|_{E_\e}\,d\si\r)^p\,ds}\\
\vs
\ds{\leq \int_0^t\le(\int_0^s (s-\si)^{-\beta-1+\e}\le|Y_1(\si\wedge \tau_R,x)-Y_2(\si\wedge \tau_R,x)\r|_{E}\,d\si\r)^p\,ds}\\
\vs
\ds{\leq \le(\int_0^ts^{-(1+\beta-\e)\frac p{p-1}}\,ds\r)^{p-1}\int_0^t\le|Y_1(s\wedge \tau_R,x)-Y_2(s\wedge \tau_R,x)\r|^p_{E}\,ds,}
\end{array}\]
so that
\begin{equation}
\label{dp254}
|I_3(t\wedge \tau_R)|^p_E\leq c_{\beta,p,\la}(t)\int_0^t\le|Y_1(s\wedge \tau_R,x)-Y_2(s\wedge \tau_R,x)\r|^p_{E}\,ds,
\end{equation}
for some function $c_{\beta,p,\la}(t)\downarrow 0$, as $t\downarrow 0$.

\medskip

{\it Step 4.} By a stochastic factorization argument, for any $R>0$ and $\beta \in\,(0,1)$ we have
\[I_4(t\wedge \tau_R)=c_\beta\int_0^{t\wedge \tau_R}(t\wedge \tau_R-s)^{\beta-1}e^{(t\wedge \tau_R-s)A}v_\beta(s)\,ds,\]
where
\[v_\beta(s)=\int_0^s (s-\si)^{-\beta}e^{(s-\si)A}\le[D\Phi(Y_1(\si,x))-D\Phi(Y_2(\si,x))\r]\cdot dw(\si).\]
Therefore, if $\e<2\beta$ and $p>2/(2\beta-\e)\vee 1/\e$ we have
\[\begin{array}{l}
\ds{|I_4(t\wedge \tau_R)|^p_E\leq |I_4(t\wedge \tau_R)|^p_{W^{\e,p}(0,1)}}\\
\vs
\ds{\leq c_{\beta,p}\le(\int_0^{t\wedge \tau_R}(t\wedge \tau_R-s)^{\beta-1-\frac \e 2}I_{\{s<\tau_R\}}|v_\beta(s)|_{L^p(0,1)}\,ds\r)^{\frac 1p}}\\
\vs
\ds{\leq c_{\beta,p}\le(\int_0^{t\wedge \tau_R}(t\wedge \tau_R-s)^{(\beta-1-\frac \e 2)\frac p{p-1}}\,ds\r)^{p-1}\int_0^tI_{\{s<\tau_R\}}|v_\beta(s)|^p_{L^p(0,1)}\,ds}\\
\vs
\ds{\leq c_{\beta,p}(t)\int_0^tI_{\{s<\tau_R\}}|v_\beta(s)|^p_{L^p(0,1)}\,ds.}
\end{array}\]
For any $\xi \in\,[0,1]$, we have
\[\begin{array}{l}
\ds{I_{\{s<\tau_R\}}|v_\beta(s,\xi)|^p}\\
\vs
\ds{\leq \le|\sum_{i=1}^\infty \int_0^s (s-\si)^{-\beta} \le(e^{(s-\si)A}\le[D\Phi(Y_1(\si\wedge \tau_R,x))-D\Phi(Y_2(\si\wedge \tau_R,x))\r]e_i\r)(\xi) \,d\beta_i(\si)\r|^p,}
\end{array}\]
then, from the Burkholder-Davies-Gundy inequality we have
\[\begin{array}{l}
\ds{\E\,I_{\{s<\tau_R\}}|v_\beta(s,\xi)|^p\leq c_p\,\E\le(\int_0^s(s-\si)^{-2\beta}\r.}\\
\vs
\ds{\le. \sum_{i=1}^\infty  \le|\le(e^{(s-\si)A}\le[D\Phi(Y_1(\si\wedge \tau_R,x))-D\Phi(Y_2(\si\wedge \tau_R,x))\r]e_i\r)(\xi)\r|^2\,d\si\r)^{\frac p 2}.}
\end{array}\]
Now, thanks to \eqref{dp242} we have
\[\begin{array}{l}
\ds{\sum_{i=1}^\infty  \le|\le(e^{(s-\si)A}\le[D\Phi(Y_1(\si\wedge \tau_R,x))-D\Phi(Y_2(\si\wedge \tau_R,x))\r]e_i\r)(\xi)\r|^2}\\
\vs
\ds{=\sum_{i=1}^\infty \le|\le<\le[D\Phi(Y_1(\si\wedge \tau_R,x))-D\Phi(Y_2(\si\wedge \tau_R,x))\r]e_i,K_{s-\si}(\xi,\cdot)\r>_H\r|^2}\\
\vs
\ds{=\le|\le[D\Phi(Y_1(\si\wedge \tau_R,x))-D\Phi(Y_2(\si\wedge \tau_R,x))\r]^\star K_{s-\si}(\xi,\cdot)\r|_H^2}\\
\vs
\ds{\leq c\,(s-\si)^{-\frac 12}\le|Y_1(\si\wedge \tau_R,x)-Y_2(\si\wedge \tau_R,x)\r|^2_E,}
\end{array}\]
and then
\[\E\,I_{\{s<\tau_R\}}|v_\beta(s)|^p_{L^p(0,1)}\leq c_p\,\E\le(\int_0^s(s-\si)^{-\le(2\beta+\frac 12\r)} \le|Y_1(\si\wedge \tau_R,x)-Y_2(\si\wedge \tau_R,x)\r|^2_E\,d\si\r)^{\frac p 2}.\]
Hence, if $\beta<1/4$ and $p>2/(1-4\beta)$, this implies
\[\begin{array}{l}
\ds{\E\int_0^tI_{\{s<\tau_R\}}|v_\beta(s)|^p_{L^p(0,1)}\,ds\leq c_{\beta,p}(t)\int_0^t\E\,\le|Y_1(s\wedge \tau_R,x)-Y_2(s\wedge \tau_R,x)\r|^p_E\,ds,}
\end{array}\]
so that
\begin{equation}
\label{dp255}
\E|I_4(t\wedge \tau_R)|^p_E\leq c_{\beta,p}(t)\int_0^t\E\,\le|Y_1(s\wedge \tau_R,x)-Y_2(s\wedge \tau_R,x)\r|^p_E\,ds.\end{equation}

\medskip

{\it Step 5. Conclusion.} From \eqref{dp221}, \eqref{dp256}, \eqref{dp254} and \eqref{dp255}, for any $R>0$ and for any  $p$ and $\la$ large enough, we have
\[\begin{array}{l}
\ds{\E\,\le|Y_1(t\wedge \tau_R,x)-Y_2(t\wedge \tau_R,x)\r|_E^p\leq \frac 12 \E\,\le|Y_1(t\wedge \tau_R,x)-Y_2(t\wedge \tau_R,x)\r|_E^p}\\
\vs
\ds{+c_{p,R}(T)\int_0^t\E\,\le|Y_1(s\wedge \tau_R,x)-Y_2(s\wedge \tau_R,x)\r|^p_E\,ds,\ \ \ \ t \in\,[0,T].}
\end{array}\]
This implies that for any fixed $R>0$
\[\E\,\le|Y_1(t\wedge \tau_R,x)-Y_2(t\wedge \tau_R,x)\r|_E^p=0.\]
Therefore, if we take the limit as $R>0$, since $\tau_R\uparrow T$, $\Pro$-a.s. as $R\uparrow +\infty$, we conclude that
\[\E\,\le|Y_1(t,x)-Y_2(t,x)\r|_E^p=0,\ \ \ \  t \in\,[0,T].\]

\end{proof}

\end{document}